\documentclass[12pt]{amsart}
\usepackage{amsmath}
\usepackage{amssymb}
\usepackage{color}
\usepackage{ifthen}
\usepackage{amscd}

\newtheorem{propo}{Proposition}[section]
\newtheorem{corol}[propo]{Corollary}
\newtheorem{theor}[propo]{Theorem}
\newtheorem{lemma}[propo]{Lemma}
\theoremstyle{definition}
\newtheorem{defin}[propo]{Definition}

\theoremstyle{remark}
\newtheorem{remar}[propo]{Remark}

\numberwithin{equation}{section}

\newcommand{\ad }{\mathrm{ad}}
\newcommand{\Ad }[2]{#1 \triangleright #2}
\newcommand{\al }{\alpha }
\newcommand{\Aut }{\mathrm{Aut}}
\newcommand{\cA }{\mathcal{A}}
\newcommand{\cC }{\mathcal{C}}
\newcommand{\cc }[1]{\mathcal{O}_{#1}}

\newcommand{\cK }{\mathcal{K}}
\newcommand{\Cm }{A}
\newcommand{\cm }{a}
\newcommand{\cM }{\mathcal{M}}
\newcommand{\co }{\mathrm{co}}
\newcommand{\coa }{\delta }
\newcommand{\copr }{\varDelta }
\newcommand{\cou }{\varepsilon }

\newcommand{\End }{\mathrm{End}}
\newcommand{\fd }{finite-dimensional}
\newcommand{\fie }{\Bbbk }
\newcommand{\fg }{\mathrm{fg}}

\newcommand{\gbhH }{ {}_H^H\mathcal{H}}
\newcommand{\gr }{\mathrm{gr}}

\newcommand{\Hom }{\mathrm{Hom}}
\newcommand{\Ib }{\mathbb{I}}
\newcommand{\icyH }{\widehat{^H_H\mathcal{YD}}}
\newcommand{\id }{\mathrm{id}}
\newcommand{\img }{\mathrm{Im\,}}
\newcommand{\io }[2]{\iota _{#1}^{#2}} % embedding of #1 into #2
\newcommand{\lact }{\cdot }
\newcommand{\mult }{\mathrm{mult}}
\newcommand{\NA }{\mathfrak{B}}
\newcommand{\ndN }{\mathbb{N}}
\newcommand{\ndZ }{\mathbb{Z}}
\newcommand{\NI }{\mathfrak{I}}
\newcommand{\Ob }{\mathrm{Ob}}
\newcommand{\ot }{\otimes }
\newcommand{\p }[1]{\langle #1 \rangle }
\newcommand{\pr }{\mathrm{pr}}
\newcommand{\re }{^\mathrm{re}}
\newcommand{\rfl }{r}
\newcommand{\rsys }{\boldsymbol{\Delta }}
\newcommand{\rsyst }{\widetilde{\boldsymbol{\Delta }}}
\newcommand{\Rwg }{\mathcal{R}}

\newcommand{\udeg }{\underline{\mathrm{deg}}}
\newcommand{\Wg }{\mathcal{W}}

\newcommand{\ydG }{ {}^G_G\mathcal{YD}}
\newcommand{\ydH }{ {}^H_H\mathcal{YD}}

\newcommand{\ydHf }{ \ydH ^{\mathrm{fd}}}
\newcommand{\ydHp }{ {}^{H'}_{H'}\mathcal{YD}}
\newcommand{\ydHtf }{ \ydH ^{\ndZ ^\theta }}
\newcommand{\ydHtkf }{ \ydH ^{\ndZ ^{\theta +\kappa }}}%_\mathrm{fd}}

\renewcommand{\_}[1]{_{(#1)}}
\renewcommand{\^}[1]{^{(#1)}}

\title[Root systems and Weyl groupoids]{Root
systems and Weyl groupoids for Nichols algebras}

\author{I. Heckenberger}
\address{Istv\'an Heckenberger,
Mathematisches Institut,
Universit\"at M\"unchen,
Theresienstr. 39,
D-80333 Munich, Germany}
\email{i.heckenberger@googlemail.com}
\author{H.-J. Schneider}
\address{
Hans-J\"urgen Schneider,
Mathematisches Institut,
Universit\"at M\"unchen,
Theresienstr. 39,
D-80333 Munich, Germany}
\email{Hans-Juergen.Schneider@mathematik.uni-muenchen.de}
\thanks{The work of I.H. was supported by DFG within a Heisenberg fellowship
at the University of Munich}

\begin{document}

\begin{abstract}
  Motivated by work of Kac and Lusztig, we define a root system and a Weyl
  groupoid for a large class of semisimple Yetter-Drinfeld modules over an
  arbitrary Hopf algebra. The obtained combinatorial structure fits perfectly
  into an
  existing framework of generalized root systems associated to a family of
  Cartan matrices, and provides novel insight into Nichols algebras. 
  We demonstrate the power of our construction
  with new results on Nichols algebras over finite non-abelian simple groups
  and symmetric groups.
\end{abstract}


\maketitle

\section*{Introduction}

In \cite{b-Kac90}, Kac defines the %Kac-Moody
Lie algebra $\mathfrak{g}(A)$ associated to a
symmetrizable Cartan matrix $A = (a_{ij})_{1 \leq i,j \leq n}$  by generators
and relations, where the relations are not given explicitly, but they are
determined by dividing out an ideal with a certain universal property.
Similarly, Lusztig \cite{b-Lusztig93} defines the braided algebra
$\mathbf{f}$, that is the plus part of the quantum deformation of the the
universal enveloping algebra of $\mathfrak{g}(A)$, by dividing out from the
free algebra an ideal defined by a universal property (the radical of a
bilinear form). After a considerable amount of work the relations turn out
to be the Serre relations.

Using the language of braided vector spaces and braided categories,
Lusztig's definition can be formulated as follows. Let $V$ be a vector space with basis $x_1, \dots,x_n$ and define a braiding
$$c : V \otimes V \to V \otimes V,\qquad
c(x_i \otimes x_j) = q^{d_ia_{ij}} x_j \otimes x_i \text{ for all $i,j$},$$
where $q$ is the deformation parameter and $(d_ia_{ij})$ is the symmetrized
Cartan matrix. The braiding $c$ has a categorical explanation. Let $G$ be the
free abelian group with basis $K_1,\dots,K_n$ and let $\ydG$ be the category
of Yetter-Drinfeld modules over $G$, that is, of $G$-graded vector spaces
which are $G$-modules where each $G$-homogeneous part is stable under the
action of $G$. Then $V \in \ydG$, where each $x_i$ has degree $K_i$ and where
the action is given by $K_i \cdot x_j = q^{d_ia_{ij}}x_j$ for all $i,j$, and $V$
is a braided vector space as an object of the braided category $\ydG$. Then
$$\mathbf{f} = T(V)/I_V,$$
where the free algebra $T(V)$ is a braided Hopf algebra such that the elements
of $V$ are primitive, and $I_V$ is the largest coideal spanned  by elements of
$\ndN $-degree $\geq 2$.
In later terminology, $\NA (V) = T(V)/I_V$ is called the
\textit{Nichols algebra of} $V$.
Note that $V= \fie{}x_1 \oplus \cdots \oplus \fie{}x_n$ is a direct sum of
one-dimensional Yetter-Drinfeld modules $\fie{} x_i$, and the Nichols algebra
of the irreducible pieces $\fie{}x_i$ is easy to compute as a (truncated)
commutatative polynomial ring.

Yetter-Drinfeld modules can be defined over any Hopf algebra $H$ over a field
$\fie$ with bijective antipode instead of the group algebra of $G$. For
details we refer to Sect.~\ref{sec:prelims}. It is a
fundamental problem in Hopf algebra theory to understand $\mathfrak{B}(V)$ for
arbitrary objects $V$ in the category $\ydH$ of Yetter-Drinfeld modules over
$H$ (see the survey article \cite{inp-AndrSchn02} of Andruskiewitsch and the
second author).
In this paper we study the Nichols algebra 
$$\NA (M_1\oplus \cdots \oplus M_\theta ),\qquad
\theta \geq 1,\,\,  M_1,\dots,M_{\theta} \in \ydH$$ 
of the direct sum of finitely many \fd{} irreducible Yetter-Drinfeld modules
$M_1,\dots,M_{\theta}$.
We assume that finite tensor powers of
$M_1\oplus \cdots \oplus M_\theta $ are semisimple in the category $\ydH $
of Yetter-Drinfeld modules. Moreover, the braided adjoint action should
satisfy a weak finiteness condition (see Def.~\ref{de:cm}).
By Prop.~\ref{pr:Tn}, the finiteness condition
can be checked directly using
the braiding of $M_1\oplus \cdots \oplus M_\theta $.
It holds for Lusztig's algebra $\mathbf{f}$.
Our main achievement (Thms.~\ref{th:Cscheme}, \ref{th:rsysC}) is to
associate to $M_1,\dots ,M_\theta $ a
generalized root system $\Rwg $
in the sense of the paper \cite{a-HeckYam08} by Yamane and the first author.
This is done  similarly to the definition
of the root system of a Kac-Moody Lie algebra $\mathfrak{g}(A)$ in
\cite{b-Kac90}. Instead of one Cartan matrix $A$ we have to deal with a family
of Cartan matrices and with the Weyl proupoid of $\Rwg $ instead of the Weyl
group. In contrast to Lusztig's approach, the Cartan matrices are not given a
priori, but they are obtained from the finiteness condition on the braided
adjoint action.

The effect for the Nichols algebra is tremendous:
Its dimension and (multivariate) Hilbert series are controlled by the set of
roots, the Nichols algebras of $M_1,\dots,M_\theta $,
and their images under reflections.
The latter are much easier to calculate than the Nichols algebra itself. 
As for Kac-Moody Lie algebras, the situation is best if all roots are real.
This is the case when the Weyl groupoid is finite, in particular when the
Nichols algebra is \fd{}. Then we show in Thm.~\ref{th:maingeneral}
the following striking result on the braided adjoint action:
For all $i\not=j$ and $m\ge 1$, the Yetter-Drinfeld module $(\ad
_cM_i)^m(M_j)$ is zero or irreducible.

To define the root system $\Rwg $,
we prove in Thm.~\ref{th:BVirred}
one of the main results in this paper:
There exists a totally ordered index set $(L,\le)$ and a family 
$(W_l)_{l\in L}$ of \fd{} irreducible $\ndN _0^\theta $-graded objects in
$\ydH $
such that
\begin{align*}
  \NA (M_1\oplus \cdots \oplus M_\theta )\simeq
  \ot _{l\in L}\NA (W_l)
%  \label{eq:BMdec}
\end{align*}
as $\ndN _0^\theta $-graded objects in $\ydH $, where $\deg N_i=\al _i$
for $1\le i\le \theta $, and $\{\al _1,\dots,\al _\theta \}$ is the standard
basis of $\ndZ ^\theta $. Moreover, the Yetter-Drinfeld modules $W_l$ and
their degrees in $\ndZ ^\theta $ are uniquely determined.
The main ingredients of the proof are a generalization by Gra{\~n}a and the
first author \cite{a-HeckGran07} of Kharchenko's result
on PBW-bases of braided Hopf algebras of diagonal type \cite{a-Khar99}, and
Thm.~\ref{th:RisNA}, which says that certain braided Hopf algebras are
isomorphic as a Yetter-Drinfeld module to a Nichols algebra.

The second key part of the construction of our root system is the definition of
the reflections and the proof of their basic properties.
This is the main result of the paper \cite{p-AHS08}
of Andruskiewitsch and the authors.

Once the existence of the root system is established, we can use results on
abstract generalized root systems recalled in
Sect.~\ref{sec:Wgrs}.
For recent work in this area
we refer to \cite{a-HeckYam08} and the paper \cite{p-CH08} by Cuntz and the
first author.

Nichols algebras appear naturally inside of the associated graded Hopf algebra
of pointed Hopf algebras \cite{a-AndrSchn98}, \cite{inp-AndrSchn02}.
Recall that a Hopf algebra is pointed if its simple
comodules are one-dimensional. In particular, Hopf algebras generated by
skew-primitive and group-like elements (as the quantum groups
$U_q(\mathfrak{g}(A)$) are pointed.
{}From our theory,
we expect a deep impact on the further analysis of pointed Hopf algebras.

Our root system generalizes the construction of the first author
\cite{a-Heck06a} in the case of diagonal braidings, where $H$ is the group
algebra of an abelian group over a field of
characteristic zero. In this setting, \fd{} Nichols algebras have been
classified in \cite{p-Heck06b}. These results allowed to complete the
classification of a large class of \fd{}
pointed Hopf algebras with abelian group of group-like
elements
\cite{a-AndrSchn05p}.

The classification of \fd{} pointed Hopf algebras with non-abelian group of
group-like elements is not known, but see \cite{a-AndrZhang07}, 
\cite{p-AndrFant08} and the references therein.
In the few examples, the group seems to be close
to be abelian. This observation is supported by our results in
Sect.~\ref{sec:app}, where the field is assumed to be algebraically closed of
characteristic zero. In particular,
in Cors.~\ref{co:sgrp}, \ref{co:Sn}
we show that the Nichols algebra of a non-simple Yetter-Drinfeld
module over any non-abelian simple group or over the symmetric group
$\mathbb{S}_n$, $n\ge 3$, is infinite dimensional.
For arbitrary finite groups $G$ we prove in Thm.~\ref{th:stst} the following
necessary condition for finiteness of the dimension of the Nichols algebra:
If $i\not=j$ and if $s$ respectively $t$ are the degrees of nonzero
homogeneous elements in $M_i$ respectively $M_j$, then
$$(st)^2=(ts)^2 \quad \text{in $G$.}$$
We believe that these rather immediate consequences
of the existence of
the root system form just the tip of the iceberg.

We would like to express our thanks
to G. Malle for providing us with information on commuting conjugacy
classes of finite groups.

\section{Preliminaries}
\label{sec:prelims}

Let $\fie $ be a field and let $H$ be a Hopf algebra over $\fie $ with
bijective antipode $S$.
Recall that a (left) Yetter-Drinfeld module over $H$ \cite[\S
10.6]{b-Montg93}, \cite{inp-AndrSchn02}
is an $H$-module $V$
equipped with a left comodule structure $\delta :V\to H\ot V$, $v\mapsto
v\_{-1}\ot v\_0$, such that
$\delta (h\cdot v)=h\_1 v\_{-1}S(h\_3)\ot h\_2\cdot v\_0$
for all $h\in H$, $v\in V$.

Let $\ydH $ and $\ydHf $ denote the
category of Yetter-Drinfeld modules and \fd{} Yetter-Drinfeld modules over
$H$, respectively. The categories $\ydH $ and $\ydHf $ are braided with
braiding 
$$c:X\ot Y\to Y\ot X,\quad c(x\ot y)=x\_{-1}\cdot y\ot x\_0$$
for all
$X,Y\in \ydH $, $x\in X$, $y\in Y$. Braided bialgebras and Hopf algebras
in this paper
are always bialgebras and Hopf algebras in the braided category $\ydH $.

If $B\in \ydH $ is a braided Hopf algebra, $x,y\in B$, and $x$ is a primitive
element, then let $(\ad _c x)(y)=xy-(x\_{-1}\cdot y)x\_0$.

%In this paper we will usually consider objects
%$V\in \ydH $ which satisfy one of the following semisimplicity assumptions:
%\begin{itemize}
%  \item[($*$)] All tensor powers $V^{\ot m}$ of $V$, where $m\in \ndN $, are
%    direct sums of irreducible objects in $\ydH $.
%  \item[($**$)] All tensor powers $V^{\ot m}$ of $V$, where $m\in \ndN $, are
%    direct sums of irreducible objects in $\ydHf $.
%\end{itemize}
%Note that if $V\in \ydH $ satisfies ($*$), then $V$ itself is a direct
%sum of irreducible objects in $\ydH $, and for each $m\in \ndN $ the
%Yetter-Drinfeld submodules of
%$\oplus _{n=1}^m V^{\ot n}$ satisfy ($*$).
%The analogous statement holds also for $V\in \ydHf $ satisfying ($**$).

Recall that a braided bialgebra is \textit{connected}, if its coradical is
$\fie 1$.
Let $\gbhH $ denote the category of connected $\ndN _0$-graded braided Hopf
algebras in $\ydH $, which are generated as an algebra by elements of degree
$1$. The morphisms of $\gbhH $ should be the maps of $\ndN _0$-graded
braided Hopf algebras. For any $R\in \gbhH $ and $n\in \ndN _0$ we write
$R(n)$ for the homogeneous component of $R$ of degree $n$.
Let $P(R)$ be the set of primitive elements of $R\in \gbhH $.
Let $\gbhH ^\fg $ be the full subcategory of $\gbhH $ consisting
of those braided Hopf algebras $R$, for which $\dim R(1)<\infty $.

\begin{defin}
  Let $R\in \gbhH $. The maximal coideal of $R$
  contained in $(R^+)^2=\oplus _{n\ge 2}R(n)$ is denoted by $\NI _R$.
  Let $\pi _R:R\to R/\NI _R$ denote the canonical projection.

  If $P(R)=R(1)$, then $R$ is called the \textit{Nichols algebra} of $R(1)$
  \cite{inp-AndrSchn02}. Then we write $R=\NA (V)$, where $V=R(1)\in \ydH $.
\end{defin}

Let $R\in \gbhH $. Then $\NI _R$ is an
$\ndN _0$-graded Hopf ideal of $R$, and
\begin{align}\label{eq:NAR}
  R/\NI _R\simeq \NA (R(1)).
\end{align}

Let $V\in \ydH $.  The tensor algebra $T(V)$ is a braided Hopf
algebra in $\ydH $ such that $P(T(V))=V$.
Another description of the ideal $\NI _{T(V)}$ is as the sum of the kernels
of the quantum symmetrizer $S_n\in \End (V^{\ot n})$, $n\ge 2$,
introduced by Woronowicz
\cite{a-Woro2}, see \cite{a-Schauen96}, \cite[10.4.13]{b-Majid1}.
Explicitly, $S_n$ can be defined with help of the quantum shuffle maps
$S_{k,1}\in \End (V^{\otimes k+1})$, $k\ge 1$,
as follows.
We write $c_{i,i+1}$, if we apply the braiding to the $i$-th and $i+1$-st
components of a tensor product of Yetter-Drinfeld modules. Then
\begin{align}
  \label{eq:shuffle}
  S_{n-1,1}=&\sum _{k=1}^n c_{n-1,n}c_{n-2,n-1}\cdots c_{k,k+1},\\
  \label{eq:Sn}
%  S_n=&\prod _{k=1}^{n-1} (S_{k,1}\ot \id ^{\ot n-k-1}).\\
  S_n=&(S_{1,1}\ot \id )(S_{2,1}\ot \id )\cdots
  (S_{n-2,1} \ot \id )S_{n-1,1}.
\end{align}

For any $\ndZ ^n$-graded vector space $X$, where $n\in \ndN $, let
$X_\gamma $ denote the homogeneous component of degree
$\gamma \in \ndZ ^\theta $ of $X$.

%Clearly, $\NA (R(1))\in \gbhH $.
%The following description of $\NI _R$ is standard, see
%e.\,g.\ \cite[Lemma\,2.1]{inp-MilSchn00}.
%Let $\NI _0(R)=0$, and for all
%$i\in \ndN _0$ define resursively
%\begin{align*}
%  \NI _{i+1}(R)=R\big(\pi _i^{-1}\big(P(R/\NI _i(R))\big)\cap (R^+)^2\big)R,
%\end{align*}
%where $\pi _i:R\to R/\NI _i(R)$ is the canonical
%map. Then $\NI _i(R)$ is an $\ndN _0$-graded braided Hopf ideal for all
%$i\in \ndN _0$, and $\cup _{n\in \ndN }\NI _i(R)=\NI _R$.
%Further, the
%construction also says that
%\begin{align*}
%  \NI _m(R/\NI _n(R))=\NI _{m+n}(R)/\NI _n(R)
%\end{align*}
%for all $m,n\in \ndN _0$.

\section{Approximations of Nichols algebras}

In this section we define and study for each $i\in \ndN \cup \{\infty \}$ a
functor $F_i: \gbhH \to \gbhH $ and a natural transformation $\lambda _i$ from
$\id $ to $F_i$. The functor $F_i$ converts primitive elements of degree
at most $i$ to elements of degree $1$, and $\lambda _i$ maps
primitive elements of degree $2,3,\dots ,i$ to zero. 

\begin{defin}
  Let $R=\oplus _{n=0}^\infty R(n)\in \gbhH $ and $i\in \ndN \cup \{\infty \}$.
  Let
\begin{gather}
%  U'_i(R)=\oplus _{n=2}^i R(n)\cap P(R),\qquad
  P_i(R)=\oplus _{n=1}^i R(n)\cap P(R),
  \label{eq:Pi}\\
  R_{-1}=0,\quad R_0=\fie , \quad
  R_k=(\fie \oplus P_i(R))^k\subset R \quad \text{for all $k\in \ndN $,}
  \label{eq:Rk}\\
  F_i(R)=\oplus _{k=0}^\infty R_k/R_{k-1}.
  \label{eq:Fi}
\end{gather}
Let $\lambda _i:R\to F_i(R)$ be the linear map defined by
\begin{align}
  \lambda _i :R(k)\ni x\mapsto x+R_{k-1}\in R_k/R_{k-1}\quad
  \text{for all $k\in \ndN _0$.}
  \label{eq:lambda}
\end{align}
For later use we define
\begin{align}
P'_i(R)=\oplus _{n=2}^i R(n)\cap P(R).
\end{align}
\end{defin}

\begin{remar}\label{re:Fi}
  Let $R=\oplus _{n=0}^\infty R(n)\in \gbhH $ and $i\in \ndN \cup \{\infty \}$.
  Then $R_k/R_{k-1}\in \ydH $ for all $k\in \ndN _0$, and $F_i(R)\in \gbhH $
  with algebra and coalgebra structure induced by those of $R$,
  and $\ndN _0$-grading given by the decomposition in Eq.~\eqref{eq:Fi}.
  If $i\in \ndN $ and $R\in \gbhH ^\fg $, then
  $\dim R(k)<\infty $ for all $k\in \ndN $, since $R(k)=R(1)^k$.
  Therefore in this case $P_i(R)$ is \fd , and hence $F_i(R)\in \gbhH ^\fg $.

  Let $R,S\in \gbhH $ and $f:R\to S$ a morphism.
  Then $f$ induces a morphism $F_i(f):F_i(R)\to F_i(S)$ in $\gbhH $.
  Thus $F_i$ is a covariant functor.

The map $\lambda _i:R\to F_i(R)$ is
well-defined, since $R(k)=R(1)^k$ and $R(1)\subset P_i(R)$,
%.
%Moreover, $\lambda _i$
%is $\ndN _0$-graded and compatible with the algebra and
%coalgebra structures of $R$ and $F_i(R)$.
and it is a morphism in $\gbhH $.
Moreover, $\ker \lambda _i$ is the ideal of $R$ generated by
$P'_i(R)$, and $\img \lambda _i$ is the subalgebra of $F_i(R)$
generated by $R(1)\subset F_i(R)(1)$.
\end{remar}
%The above remark and the following statement make the functors $F_i$,
%where $i\in \ndN \cup \{\infty \}$, more interesting.

\begin{propo}
  Let $R\in \gbhH $ and $i\in \ndN \cup \{\infty \}$.
  Let $\fie \langle R(1)\rangle $ be the subalgebra of $F_i(R)$ generated by
  $R(1)\subset F_i(R)(1)$, and let $(P'_i(R))$ be the ideal of $F_i(R)$
  generated by $P'_i(R)\subset F_i(R)(1)$.
  Then $\fie \langle R(1)\rangle \in \gbhH $,
  $(P'_i(R))$ is an $\ndN _0$-graded Hopf ideal of $F_i(R)$ in $\ydH $, and
  $F_i(R)=\fie \langle R(1)\rangle \oplus (P'_i(R))$.
  \label{pr:FiR}
\end{propo}

\begin{proof}
  Clearly, $\fie \oplus R(1)$ is a subcoalgebra and Yetter-Drinfeld submodule
  of $F_i(R)$, and hence $\fie \langle R(1)\rangle \in \gbhH $. Further,
  $(P'_i(R))\in \ydH $ since $P'_i(R)\in \ydH $, and $(P'_i(R))$ is an $\ndN
  _0$-graded coideal of $F_i(R)$
  since $P'_i(R)$ is an $\ndN _0$-graded coideal of $F_i(R)$.
  The algebra $F_i(R)$ is generated by $R(1)\oplus P'_i(R)$, and
  therefore $\fie \langle R(1)\rangle +(P'_i(R))=F_i(R)$.
  It remains to show that the
  sum is direct. Since $\fie \langle R(1)\rangle $ and $(P'_i(R))$ are $\ndN
  _0$-graded, it suffices to consider $\ndN _0$-homogeneous components. Let
  $k\ge 0$ and $\bar{x}\in \fie \langle R(1)\rangle \cap (P'_i(R))\cap F_i(R)(k)$.
  Since $\bar{x}\in \fie \langle R(1)\rangle $, there exists a representant
  $x\in R(k)$ of $\bar{x}\in R_k/R_{k-1}$. Further, $\bar{x}\in (P'_i(R))$, and hence
  the set $x+R_{k-1}$
  contains an element $y$ of the ideal $(P'_i(R))\subset R$. By subtracting
  elements of $R_{k-1}$ from $y$ we may assume that $y$ is a linear
  combination of monomials containing $k$ factors from $R(1)\cup P'_i(R)$, and
  in each monomial at least one of the factors is from $P'_i(R)$. But
  $P'_i(R)\subset \oplus _{j\ge 2}R(j)$, and hence $y\in \oplus _{j\ge
  k+1}R(j)$. Since $x\in R(k)$ and $x-y\in R_{k-1}$,
  this implies that $x=y=0$ in $R_k/R_{k-1}$, and hence $\fie
  \langle R(1)\rangle \cap (P'_i(R))=0$.
\end{proof}

\begin{remar}
  Let $R\in \gbhH $ and let $R(1)=V\oplus W$ be a decomposition in $\ydH $.
  Let $\fie \langle V\rangle $ be the subalgebra of $R$ generated by
  $V$, and let $(W)$ be the ideal of $R$
  generated by $W$. Then in general the sum $R=\fie \langle V\rangle +
  (W)$ is not direct.
  
  For example, let $H=\fie (\ndZ /2\ndZ )^2$
  and assume that $g,h\in (\ndZ /2\ndZ )^2$
  generate $(\ndZ /2\ndZ )^2$ as a group.
  Let $V,W\in \ydH $ with $V=\Bbbk x$, $W=\Bbbk y$, and
  assume that the coaction and the action of $H$ satisfy
  \begin{gather*}
    \coa (x)=g\ot x,\quad \coa (y)=h\ot y,\\
    g\lact x=-x,\quad g\lact y=y,\quad
    h\lact x=x,\quad h\lact y=-y.
  \end{gather*}
  Then $\NI =(xy-yx,x^2-y^2)$ is an
  $\ndN _0$-graded Hopf
  ideal in $T(V\oplus W)$, where $V\oplus W\subset P(T(V\oplus W))$.
  Since $x^2\in \fie \langle V\rangle \cap (W)\subset T(V\oplus W)/\NI $,
  the sum $\fie \langle V\rangle +(W)$ in $T(V\oplus W)/\NI $ is not
  direct.
\end{remar}

\begin{lemma}
  Let $R\in \gbhH $, $k\in \ndN _0$, and $i\in \ndN $ such that $i>k$.
  If $R(m)\cap P(R)=0$ for all $m$ with $2\le m\le k$, then
  $\lambda _i(R)(m)\cap P(\lambda _i(R))=0$ for all $m$ with $2\le m\le k+1$.
  \label{le:noprim}
\end{lemma}

\begin{proof}
  The condition $R(m)\cap P(R)=0$ for all $m$ with $2\le m\le k$ is equivalent
  to $P'_i(R)=\oplus _{n=k+1}^iR(n)\cap P(R)$.
  Since $\lambda _i(R)\simeq R/\ker \lambda _i$, the claim follows from
  the end of Rem.~\ref{re:Fi}.
\end{proof}

\begin{propo}\label{pr:lambda}
  Let $m\in \ndN $ and $i_2,i_3,\ldots ,i_m\in \ndN $
  such that $i_n\ge n$ for all $n\le m$. Let $R\in \gbhH $ and
  $$\lambda =\lambda _{i_m} \cdots \lambda _{i_3}\lambda _{i_2}:R\to 
  F_{i_m}\cdots F_{i_3}F_{i_2}(R).$$
  Then $\ker \lambda \cap R(k)=\NI _R \cap R(k)$ for all $k$ with $0\le k\le
  m$.
%  In particular, $\lambda _i:
%  \lambda (R(k))\to \lambda _i \lambda (R(k))$
%  is an isomorphism for all $i\in \ndN $ and all $k\le m$. 
\end{propo}

\begin{proof}
  By Lemma~\ref{le:noprim} and induction on $m$ one obtains that $\lambda
  (R(k))$ has, if $2\le k\le m$, no primitive elements besides $0$.
  Since $R$ is generated by $R(1)$ and
  $\lambda (R)\simeq R/\ker \lambda $, we obtain that
  $\ker \lambda \cap R(k)=\NI _R\cap R(k)$ for $0\le k\le m$.
  Now, if $i\le m$, then $P'_i(F_{i_m}\cdots F_{i_3}F_{i_2}(R))\cap
  \lambda (R)=0$, and hence $\lambda _i:\lambda (R)\to \lambda _i\lambda (R)$
  is an isomorphism. Otherwise $i>m$, and the claim follows from the
  equalities
  $$\ker \lambda \cap R(k)=\ker \lambda _i\lambda \cap R(k)=\NI _R\cap R(k),
  \qquad 0\le k\le m,$$
  shown in the first part of the proof.
\end{proof}

\section{Braided Hopf algebras and Yetter-Drinfeld modules}

In Thm.~\ref{th:RisNA} we give a criterion for
a braided Hopf algebra to be isomorphic
to a Nichols algebra in $\ydH $.
First we show that $F_i(R)\simeq R$ in $\ydH $ for all $R\in \gbhH $
which are semisimple objects in $\ydH $.

%Prop.~\ref{pr:R=BW} and Thm.~\ref{th:RisNA}.
%We consider two settings in which a braided Hopf algebra
%is isomorphic, as a Yetter-Drinfeld module, to a Nichols algebra, see
%Prop.~\ref{pr:R=BW} and Thm.~\ref{th:RisNA}.

\begin{lemma}
  Let $R$ be a coalgebra in $\ydH $ and let $R=\cup _{n=0}^\infty R_n$ be a
  coalgebra filtration of $R$ in $\ydH $.
  % such that $\copr R_n\subset \sum _{j=0}^nR_j\ot R_{n-j}$
  % for all $n\in \ndN _0$.
  Suppose that $R_n$ is a semisimple object in $\ydH $ for all $n\in \ndN _0$.
  Let $\gr \,R=\oplus _{n=0}^\infty R(n)$,
  where $R_{-1}=0$ and $R(n)=R_n/R_{n-1}$ for all $n\in \ndN _0$. Then the
  coalgebra $\gr \,R$ is a semisimple object in $\ydH $, and $\gr \,R\simeq R$
  as objects in $\ydH $.
  \label{le:ydiso1}
\end{lemma}

\begin{proof}
%  By assumption $R_n\in \ydH $ for all $n\in \ndN _0$.
  For each $n\in \ndN _0$ the sequence
  \begin{align*}
    0\to R_{n-1}\to R_n\to R(n)\to 0
  \end{align*}
  in $\ydH $ is exact.
  Thus $R_n\simeq R(n)\oplus R_{n-1}$ in $\ydH $ and
  $R(n)$ is semisimple for all $n\in \ndN _0$, since
  $R_n$ is semisimple for all $n\in \ndN _0$.
  By induction one obtains for all $n\in \ndN _0$ that
  $R_n\simeq \oplus _{i=0}^n R(i)$ in $\ydH $.
  For all $n\in \ndN _0$ the isomorphism
  $R_n\simeq \oplus _{i=0}^n R(i)$ can
  be extended to an isomorphism $R_{n+1}\simeq \oplus _{i=0}^{n+1} R(i)$,
  since $R_{n+1}\simeq R_n\oplus R(n+1)$.
  Thus the lemma follows from
  \begin{align*}
    {R=\cup _{n\in \ndN } R_n\simeq \cup _{n\in \ndN } \oplus _{i=0}^n R(i)
    \simeq \oplus _{i=0}^\infty R(i)=\gr \,R. \hfill \popQED \qed}
  \end{align*}
\end{proof}

\begin{lemma}
  Let $R$ be an algebra in $\ydH $ which is generated by a subspace $U\in
  \ydH $.
  Let $R_{-1}=0$, $R_0=\fie $,
  $R_n=(U+\fie)^n\subset R$ for all $n\in \ndN $,
  and $\gr \,R=\oplus _{n=0}^\infty R_n/R_{n-1}$.
  If $R_n$ is semisimple in $\ydH $ for all $n\in \ndN _0$, then
  $\gr \,R\simeq R$ as objects in $\ydH $.
  \label{le:ydiso2}
\end{lemma}

\begin{proof}
  Since $\fie \ot W\simeq W$ for all $W\in \ydH $, $R_{n-1}$ is a subobject of
  $R_n$ in $\ydH $ for all $n\in \ndN _0$.
  Thus the arguments in the proof of Lemma~\ref{le:ydiso1} yield the claim.
\end{proof}

\begin{corol}\label{co:FRsemi}
  Let $R\in \gbhH $ and $i\in \ndN $. If $R(n)$
  is semisimple in $\ydH $ for all $n\in \ndN $, then
  $F_i(R)(n)$ is semisimple in $\ydH $ for all $n\in \ndN $,
  and $F_i(R)\simeq R$ as objects in $\ydH $.
\end{corol}

\begin{proof}
  This is clear by the definition of $F_i$ and by Lemma~\ref{le:ydiso2}.
\end{proof}

Recall that for any embedding $\io{U}{V}:U\subset V$ in $\ydH $ there
exists a canonical embedding $\NA (\io{U}{V}):\NA (U)\to \NA (V)$ of braided
Hopf algebras in
$\ydH $ such that $\NA (\io{U}{V})|_U=\io{U}{V}$.

%\begin{proof}
%  First note that $I_i$ is an $\ndN _0$-graded object in $\ydH $ for all $i\in
%  \ndN _0$.
%  By definition, $I_0\subset \oplus _{n\ge 2}R(n)$, hence
%  $R(0)=\fie $ implies that $P(F^1(R)(2))=0$. Then
%  $I_1\subset \oplus _{n\ge 3}F^1(R)(n)$ and $P(F^2(R)(n))=0$ for $n\in
%  \{2,3\}$. By induction it follows that $I_i\subset \oplus _{n\ge
%  i+2}F^i(R)(n)$ and $P(F^i(R)(n))=0$ for $2\le n\le i+1$. Thus
%  $\vec{R}:=\lim _{i\to \infty }F^i(R)$ is a connected $\ndN _0$-graded
%  braided Hopf algebra in $\ydH $ generated by $\vec{R}(1)$, and
%  $P(\vec{R})=\vec{R}(1)$. Then the claim follows e.\,g.\ from
%\end{proof}

\begin{lemma}
  Let $R\in \gbhH $ and $i\in \ndN $ such that all finite tensor powers of
  $R(1)$ are semisimple in $\ydH $.
  Let $\iota _R:\NA (R(1))\to R$ be an
  $\ndN _0$-graded splitting of the exact sequence
\begin{align*}
  0 \to \NI _R \longrightarrow R \overset{\pi _R}{\longrightarrow }
  \NA (R(1)) \to 0
%  \begin{CD}
%  0 & @>>> & \NI _R & @>>> & R & @>>> & \NA (R(1)) & @>>> & 0,
%  \\
%  0 & @>>> & \NI _{\gr \,R} & @>>> & \gr \,R & @>>> & \NA (U) & @>>> & 0
%  \end{CD}
\end{align*}
  in $\ydH $.
  Then there is an isomorphism $\Phi :R\to F_i(R)$ in $\ydH $ and an
  $\ndN _0$-graded splitting $\iota _{F_i(R)}:\NA (P_i(R))\to F_i(R)$
  of the exact sequence
\begin{align*}
  0 \to \NI _{F_i(R)} \longrightarrow F_i(R) \overset{\pi
  _{F_i(R)}}{\longrightarrow } \NA (P_i(R)) \to 0
\end{align*}
  in $\ydH $ with the following properties.
  \begin{enumerate}
    \item $\Phi (R(n))\subset \oplus _{k=0}^n F_i(R)(k)$
      for all $n\in \ndN _0$.
    \item Let $\gr \,\Phi \in \Hom (R,F_i(R))$ such that
      $\gr \,\Phi |_{R(n)} = \pr _n\circ \Phi |_{R(n)}$ for
      all $n\in \ndN _0$, where $\pr _n:F_i(R)\to F_i(R)(n)$ is the canonical
      projection. Then $\gr \,\Phi =\lambda _i$.
    \item The following diagram is commutative.
  \begin{equation}
    \begin{CD}
      \NA (R(1)) & @>\iota _R>> & R\\
      @V{\NA (\io{R(1)}{P_i(R)})}VV & & @VV\Phi V \\
      \NA (P_i(R)) & @>\iota _{F_i(R)}>> & F_i(R).
    \end{CD}
    \label{eq:cdR}
  \end{equation}
  \end{enumerate}
  \label{le:moregrR}
\end{lemma}

\begin{proof}
  First we define the map $\Phi |_{R(k)}$ for all $k\in \ndN _0$ and prove
  that it satisfies (1) and (2).
  Then we show that $\Phi $ is bijective, and that the required section
  $\iota _{F_i(R)}$ of $\pi _{F_i(R)}$, which makes Diagram~\eqref{eq:cdR}
  commutative, exists.

  For all $k\in \ndN _0$ let $(P'_i(R))(k)=(P'_i(R))\cap R(k)$, where
  $(P'_i(R))\subset R$ is the ideal of $R$ generated by $P'_i(R)$.
  Since $P'_i(R)\subset \oplus _{n=2}^\infty R(n)$ is an $\ndN _0$-graded
  object in $\ydH $, we obtain that
  $(P'_i(R))=\oplus _{k=2}^\infty (P'_i(R))(k)$ and that
  $(P'_i(R))(k)\in \ydH $ for all $k\in \ndN _0$.
  Further, the sum
  \begin{align}
    \iota _R(\NA ^k(R(1)))+(P'_i(R))(k)\subset R(k)
    \label{eq:ioplusP}
  \end{align}
  is direct.
  Indeed, $\pi _R$ is an algebra map, $\pi _R(P'_i(R))=0$, and
  $\pi _R\iota _R=\id _{\NA (R(1))}$. Hence, if $x\in \NA ^k(R(1))$ and $\iota
  _R(x)\in (P'_i(R))(k)$, then $\pi _R(x)\in \pi _R( (P'_i(R)) )=0$ and
  $\pi _R(x)=\pi _R\iota _R(x)=x$, and hence $x=0$.

  For all $j,k\in \ndN $ let $V^{i k}_j\in \ydH $ such that
  $$R(k)=(R_j\cap R(k))\oplus V^{i k}_j$$
  (see Eq.~\eqref{eq:Rk}). The $V^{i k}_j$ exist, since $R(k)$ is semisimple
  in $\ydH $.
  The relation $P'_i(R)\subset \oplus _{n=2}^\infty R(n)$ implies that
  $$(P'_i(R))(k)=R_{k-1}\cap (P'_i(R))(k)=R_{k-1}\cap R(k),
    %=\cup _{j=1}^{k-1}R_j\cap (P'_i(R))(k).
    %=\cup _{j=1}^{k-1}R_j\cap R(k).
  $$
  and hence by Eq.~\eqref{eq:ioplusP} we may assume that
  $\iota _R(\NA ^k(R(1))) \subset V^{ik}_j$ if $j<k$.
  Let $k\in \ndN _0$. We define $\Phi _k:R(k)\to F_i(R)$ by setting
  \begin{equation}
    \begin{aligned}
    &\Phi _k:=\oplus _{j=0}^k\Phi _{k j},\qquad
    \Phi _{k j}: R(k)\to F_i(R)(j)=R_j/R_{j-1},\\
    &\Phi _{k j}(x):=
    \begin{cases}
      0 & \text{for $x\in V^{i k}_j$,}\\
      x+R_{j-1} & \text{for $x\in R_j\cap R(k)$,}
    \end{cases}
  \end{aligned}
    \label{eq:Phidec}
  \end{equation}
  where $0\le j\le k$. Then it is clear that (1) holds,
  and since $\gr \,\Phi |_{R(k)}=\Phi _{k
  k}$, we also get (2) by the definition of $\lambda _i$ in
  Eq.~\eqref{eq:lambda}.

  Now we prove that $\Phi $ is bijective.
  Suppose first that $x=\sum _{k=0}^n x_k$, where $x_k\in R(k)$ for all $k\le
  n$, and $\Phi (x)=0$. Then
  $$\Phi (x)=\sum _{j=0}^n\sum _{k=j}^n \Phi _{k j}(x_k)
  \in \sum _{j=0}^n F_i(R)(j)$$
  by definition of $\Phi $ and by (1), and hence
  $\sum _{k=j}^n \Phi _{k j}(x_k)=0$ for all $j\in \{0,1,\ldots ,n\}$.
  Using the definition of $\Phi $ again and again, and the facts that
  $R(k)\subset R_k$ and $R_k$ is $\ndN _0$-graded for all $k\in \ndN _0$,
  we conclude now as follows.
  \begin{align*}
    &\Phi _{n n}(x_n)=0 \quad \Rightarrow \quad x_n\in R_{n-1},\\
    &\Phi _{n-1\,n-1}(x_{n-1})+\Phi _{n\,n-1}(x_n)=0 \quad \Rightarrow
      x_{n-1}+x_n\in R_{n-2} \quad \Rightarrow \\
      &\qquad x_{n-1},x_n\in R_{n-2},\\
    &\Phi _{n-2\,n-2}(x_{n-2})+\Phi _{n-1\,n-2}(x_{n-1})+\Phi _{n\,n-2}(x_n)=0
      \quad \Rightarrow \\
      &\qquad x_{n-2}+x_{n-1}+x_n\in R_{n-3} \Rightarrow 
      x_k\in R_{n-3} \quad \text{if $n-2\le k\le n$, \dots }\\
      &\qquad \sum _{k=1}^n x_k\in R_0=\fie \Rightarrow
      x_k=0 \quad \text{if $1\le k\le n$,}\\
    &\Phi _{0 0}(x_0)=0 \quad \Rightarrow x_0\in R_{-1}=0.
  \end{align*}
  Thus $x_k=0$ for all $k\in \{0,1,\ldots ,n\}$, and hence $\Phi $ is
  injective.
  The surjectivity of $\Phi $ follows immediately from the decomposition
  \begin{align*}
    R_j/R_{j-1}=&\big((R(1)\oplus P'_i(R))^j+R_{j-1}\big)/R_{j-1}\\
    =&\oplus _{k=j}^\infty (R(k)\cap R_j+R_{j-1})/R_{j-1}\\
    =&\oplus _{k=j}^\infty (R(k)\cap R_j+R(k)\cap R_{j-1})/(R(k)\cap
    R_{j-1})
  \end{align*}
  and the definition of $\Phi $.

  Now we prove (3). Note that the canonical map $\pi _{F_i(R)}: F_i(R)\to \NA
  (P_i(R))$ is compatible with the decomposition
  $F_i(R)=\fie \langle R(1)\rangle \oplus (P'_i(R))$, see Prop.~\ref{pr:FiR},
  in the sense that
  $\pi _{F_i(R)}(\fie \langle R(1)\rangle ) \subset \fie \langle R(1)\rangle $
  and
  $\pi _{F_i(R)}\big( (P'_i(R)) \big) \subset (P'_i(R))$.
  Let $\pi ': \fie \langle R(1)\rangle \to \fie \langle R(1)\rangle $ and
  $\pi '': (P'_i(R))\to (P'_i(R))$ denote the components of $\pi _{F_i(R)}$,
  that is, $\pi _{F_i(R)}=\pi '\oplus \pi ''$.
  Let $\iota ''_{F_i(R)}: (P'_i(R))\to (P'_i(R))$
  be an arbitrary section of $\pi ''$ in $\ydH $, and define
  $$\iota '_{F_i(R)}: \fie \langle R(1)\rangle 
  \to \fie \langle R(1)\rangle ,\qquad
  \iota '_{F_i(R)}:= \lambda _i\circ \iota _R,
  $$
  where the domain of $\iota '_{F_i(R)}$ is a subobject of $\NA (P'_i(R))$,
  and the range of $\iota '_{F_i(R)}$ is a subobject of $F_i(R)$.
  Let $\iota _{F_i(R)}:=\iota '_{F_i(R)}\oplus \iota ''_{F_i(R)}$.
  Note that the choice of
  $\iota '_{F_i(R)}$ is necessary. Indeed, in order to make
  Diagram~\eqref{eq:cdR} commutative, we need precisely that
  $\iota '_{F_i(R)}= \Phi \circ \iota _R$. However,
  Eq.~\eqref{eq:Phidec} and the assumption
  $\iota _R(\NA ^k(R(1))) \subset V^{ik}_j$ for $j<k$
  imply that $\Phi _{k j}\iota _R(\NA ^k(R(1)))=0$
  for $j<k$, and hence
  $\Phi \iota _R=(\gr \,\Phi )\iota _R=\lambda _i\iota _R$.

  It remains to show that
  $\pi _{F_i(R)}\iota '_{F_i(R)}(\fie \langle R(1)\rangle )=\id |_{\fie
  \langle R(1)\rangle }$. Equivalently, we have to prove that the diagram
  \begin{equation*}
    \begin{CD}
      \NA (R(1)) & @>\iota _R>> & R\\
      @V{\NA (\iota _{R(1)}^{P_i(R)})}V\simeq V & & @VV\lambda _i V \\
      \NA (P_i(R))\supset \fie \langle R(1)\rangle & @<\pi _{F_i(R)}<< &
      \fie \langle R(1)\rangle \subset F_i(R)
    \end{CD}
  \end{equation*}
  is commutative. Clearly, this is equivalent to the commutativity of
  the diagrams
  \begin{equation*}
    \begin{CD}
      \NA (R(1)) & @>\iota _R>\simeq > & \iota _R(\NA (R(1)))\\
      @V{\NA (\iota _{R(1)}^{P_i(R)})}VV & & @VV\lambda _i V \\
      \NA (P_i(R)) & @<\pi _{F_i(R)}<< & F_i(R),
    \end{CD}
    \qquad
    \begin{CD}
      \NA (R(1)) & @<\pi _R<\simeq < & \iota _R(\NA (R(1)))\\
      @V{\NA (\iota _{R(1)}^{P_i(R)})}VV & & @VV\lambda _i V \\
      \NA (P_i(R)) & @<\pi _{F_i(R)}<< & F_i(R),
    \end{CD}
  \end{equation*}
  where the second one is obtained from the first one by using
  that $\iota _R$ is a section of $\pi _R$.
  The second diagram is easily seen to be commutative.
  Indeed, the diagram
  \begin{equation*}
    \begin{CD}
      \NA (R(1)) & @<\pi _R<< & R\\
      @V{\NA (\iota _{R(1)}^{P_i(R)})}VV & & @VV\lambda _i V \\
      \NA (P_i(R)) & @<\pi _{F_i(R)}<< & F_i(R),
    \end{CD}
  \end{equation*}
  is a diagram in $\gbhH $, and it is commutative, since it is commutative
  on the generators.
\end{proof}

\begin{theor}\label{th:RisNA}
  Let $R$ be a connected braided Hopf algebra which is generated as
  an algebra by $U\subset P(R)$, where $U\in \ydH $,
  and assume that all tensor powers of $U$ are semisimple in $\ydH $.
  Then there is a subobject $V\subset R$ in $\ydH $
  such that $U\subset V$, all tensor powers of $V$ are semisimple in $\ydH $,
  and $R\simeq \NA (V)$ in $\ydH $.
\end{theor}

\begin{proof}
  It is sufficient to prove the theorem for $R\in \gbhH $ and $U=R(1)$.
  Indeed, define $\gr \,R=\oplus _{n=0}^\infty R(n)$ as in
  Lemma~\ref{le:ydiso2}. Since $U\subset P(R)$, we get $\copr : R(n)\to
  \oplus _{k=0}^n R(k)\otimes R(n-k)$, and hence $\gr \,R\in \gbhH $.
  Moreover, $\gr \,R\simeq R$ in $\ydH $ by Lemma~\ref{le:ydiso2}.
  
  Define inductively $R\^n \in \gbhH $ for all $n\in
  \ndN $ by letting $R\^1=R$ and $R\^n=F_{n-1}(R\^{n-1})$ for all
  $n\in \ndN _{\ge 2}$, and let $P_n=P_n(R\^n)$, $P'_n=P'_n(R\^n)$ for all
  $n\in \ndN $. 

  Using Corollary~\ref{co:FRsemi} and the formula $R\^n (k)=R\^n (1)^k$ for
  all $k,n\in \ndN $,
  we may apply Lemma~\ref{le:moregrR}. Hence for each
  $n\in \ndN $ there exists an isomorphism $\Phi _n: R\^n\to
  R\^{n+1}$ which satisfies the properties in
  Lemma~\ref{le:moregrR}(1)--(3). In particular, for each $n\in \ndN _{\ge 2}$
  the diagrams
  \begin{equation}
    \begin{CD}
      \NA (P_{n-1}) & @>\iota _{R\^n}>> & R\^n\\
      @V\NA (\io{P_{n-1}}{P_n})VV & & @VV\Phi _n V \\
      \NA (P_n) & @>>\iota _{R\^{n+1}}> & R\^{n+1},
    \end{CD}\qquad
    \begin{CD}
      \NA (P_{n-1}) & @>\Phi _1^{-1}\cdots \Phi _{n-1}^{-1}\iota _{R\^n}>> & R\\
      @V\NA (\io{P_{n-1}}{P_n})VV & & @| \\
      \NA (P_n) & @>>\Phi _1^{-1}\cdots \Phi _n^{-1}\iota _{R\^{n+1}}> & R
    \end{CD}
    \label{eq:commtri}
  \end{equation}
  commute.
  Hence there is a map $$\Psi :\NA (\cup _{n=1}^\infty P_n)
  =\cup _{n=1}^\infty \NA (P_n)\to R$$
  in $\ydH $, and since $\iota _{R\^n}$ is injective for all $n\in \ndN $,
  the map $\Psi $ is injective.
  Let $V=\cup _{n=0}^\infty P_n\in \ydH $.
  Since for all $n\in \ndN $ all tensor powers of $P_n$ are semisimple in
  $\ydH $, all tensor powers of $V$ are semisimple in $\ydH $. We are left to
  show that $\Psi $ is surjective.

  For all $n\in \ndN $ let $\varphi _n=\Phi _{n-1}\cdots \Phi _2:
  R\^2=R\to R\^n$.
  By Lemma~\ref{le:moregrR}, for all $x\in R$ there exist $k\in \ndN _0$
  and $m\in \ndN _{\ge 2}$ such that
  \begin{gather*}
    \varphi _m(x)\in \oplus _{n=0}^k R\^m (n),\quad
    \varphi _t(x)\notin \oplus _{n=0}^{k-1} R\^t (n)
  \end{gather*}
  for all $t\ge m$.
  We prove by induction on $k$ that $x\in \Psi (\NA (V))$.
  First, let $k=0$. Then $x\in \fie = \NA ^0(V)$, and $\Psi (x)=x$
  by definition of $\Psi $.

  Let now $k\ge 1$. Since
  $$\varphi _m(x)-\iota _{R\^m}\pi _{R\^m}\varphi _m(x)\in
  R\^m(k)\cap \ker \pi _{R\^m} =R\^m(k)\cap \NI _{R\^m},$$
  Prop.~\ref{pr:lambda} and Lemma~\ref{le:moregrR}(2)
  imply that
  $$\Phi _{m+k-2}\cdots \Phi _{m+1}\Phi _m(\varphi _m(x)
  -\iota _{R\^m}\pi _{R\^m}\varphi _m(x))
  \in \oplus _{n=0}^{k-1}R\^{m+k-1}(n).$$
  Now $\pi _{R\^m}\varphi _m(x)\in \NA (V)$ by definition of $\NA (V)$, and
  hence
  $$\varphi _m^{-1}\iota _{R\^m}\pi _{R\^m}\varphi _m(x)\in \Psi (\NA (V)).$$
  Further,
  $x-\varphi _m^{-1}\iota _{R\^m}\pi _{R\^m}\varphi _m(x)\in \Psi (\NA (V))$
  by induction hypothesis on $k$. Thus $x\in \Psi (\NA (V))$.
\end{proof}

\begin{remar}
  Assume that $R\in \ydHf $ is a connected braided Hopf algebra and a
  semisimple object in $\ydHf $.
  Then,
  using the algorithm in \cite{a-AndrSchn02} below Cor.~7.8,
  it can be shown that $R\simeq \NA (V)$ in $\ydHf $
  for some $V\in \ydHf $.
\end{remar}

\section{Decompositions of Nichols algebras into tensor products}

Our main result in this section is Thm.~\ref{th:BVirred} giving a
decomposition of a large class of Nichols algebras.
For the proof we use Thm.~\ref{th:RisNA} and a result
from \cite{a-HeckGran07} which is based on the ideas of Kharchenko
on PBW bases of braided Hopf algebras with diagonal braiding
\cite{a-Khar99}.

%For the formulation and proof of the theorem we have to introduce
%some notation.

Let $\ydHtf $ denote
the category of $\ndZ ^\theta $-graded Yetter-Drinfeld modules over $H$ having
\fd{} homogeneous components.

\begin{remar} \label{re:Heq}
  Let $H'$ denote the Hopf algebra $\fie \ndZ ^\theta \otimes H$.
  For any $\ndZ ^\theta $-graded object $V\in \ydH $ let $V'\in \ydHp $
  such that $V'=V$ as an $H$-module, $\gamma v=v$ for all $v\in V'$ and
  $\gamma \in \ndZ ^\theta $, and the left coaction on $V'$ is determined by
  $\delta _{V'}(v)=\gamma v_{(-1)}\otimes v_{(0)}$
  for all $v\in V_\gamma \subset V'$,
  where $v_{(-1)}\ot v_{(0)}$ is the coaction of $H$ on $v\in V$.
  This way the category of $\ndZ ^\theta $-graded objects in $\ydH $ is
  equivalent to the full subcategory of $\ydHp $ consisting of those
  Yetter-Drinfeld modules, for which the action of $\fie \ndZ ^\theta $ is
  trivial.
\end{remar}

Let $\cA =\{\al _1,\ldots ,\al _\theta \}$ be a fixed basis of $\ndZ ^\theta $.
For $\gamma =\sum _{i=1}^\theta n_i\al _i\in \ndZ ^\theta $ let
$|\gamma |_\cA =\sum _{i=1}^\theta n_i$.
For any $\ndZ ^\theta $-graded object $X$ and any $k\in \ndZ $ let
\[ X(k)=\mathop{\oplus }_{\gamma \in \ndZ ^\theta ,|\gamma |_\cA =k}
X_\gamma .\]
Clearly, if $X\in \ydHtf $ such that $X_\gamma =0$ for $\gamma \in
\ndZ ^\theta \setminus \ndN _0\cA $, where
$\ndN _0\cA =\sum _{i=1}^\theta \ndN _0\al _i$,
then $X(k)=0$ for $k<0$ and $X(k)$ is
\fd{} for all $k\ge 0$.

Let $(I,\le )$ be a totally ordered index set, and for each $i\in I$ let
$X_i$ be a connected braided Hopf algebra in $\ydH $.
Let $1$ denote the unit of (each) $X_i$.
For any finite subset $J=\{j_1<j_2<\ldots <j_t\}\subset I$
let $X_J=X_{j_1}\ot X_{j_2}\ot \cdots \ot X_{j_t}$.
The family of objects $X_J$ forms a direct system
with respect to the inclusion in the following sense.
If $J\subset K\subset I$ are finite subsets, where $J=\{j_1<j_2<\ldots <j_t\}$,
$K=\{k_1<k_2<\ldots <k_s\}$, then $X_J\subset X_K$ via the
embedding
\begin{align*}
  X_J\to X_K,\qquad
  x_{j_1}\ot x_{j_2}\ot \cdots \ot x_{j_t}\mapsto
  y_{k_1}\ot y_{k_2}\ot \cdots \ot y_{k_s},
\end{align*}
where for all $l$ with $1\le l\le s$ let
\begin{align*}
  y_{k_l}=
  \begin{cases}
    x_{j_m} & \text{if $k_l=j_m$ for some $m\le t$,}\\
    1 & \text{otherwise.}
  \end{cases}
\end{align*}
We write $\otimes _{i\in I} X_i$ for the limit of this direct system.

\begin{remar}
  If $X_i$ is $\ndZ ^\theta $-graded for all $i\in I$, then
  $\otimes _{i\in I} X_i$ is $\ndZ ^\theta $-graded. If
  $(X_i)_0=\fie $ and $X_i=\oplus _{\gamma \in \ndN _0\cA }(X_i)_\gamma $ for
  all $i\in I$,
  then
  $\otimes _{i\in I} X_i\in \ydHtf $ if and only if
  $X_i\in \ydHtf $ for all $i\in I$ and for all $\gamma \in \ndN _0\cA
  \setminus \{0\}$
  the set $\{i\in I\,|\,(X_i)_\gamma \not=0\}$ is finite.
\end{remar}

\begin{theor}\label{th:GH}
  Let $R$ be a connected braided Hopf algebra in $\ydH $, $\kappa \in \ndN $,
  and let $V_1,\ldots ,V_\kappa \in \ydH $ be
  subobjects of $P(R)$, such that $R$ is generated by $\oplus _{i=1}^\kappa
  V_i$. Let $\{e_1,\ldots ,e_\kappa \}$ be the standard basis of
  $\ndZ ^\kappa $,
  and assume that there is an $\ndN _0^\kappa $-grading of $R$ such that
  $\deg V_i=e_i$ for all $i\in \{1,\ldots ,\kappa \}$.
  If for all $\ndN _0^\kappa $-graded subalgebras $T\subset R$ and all $\ndN
  _0^\kappa $-graded ideals $T'$ of $T$ there is an $\ndN _0^\kappa $-graded
  splitting $T/T'\to R$ in $\ydH $,
  then there exists a totally ordered index set
  $(L,\le )$ and a family $(R_l)_{l\in L}$ of connected
  $\ndN _0^\kappa $-graded
  braided Hopf algebras $R_l\in \ydH $, such that
  \begin{enumerate}
    \item $\{1,\ldots ,\kappa \}\subset L$, and for $l\in \{1,\ldots
      ,\kappa \}$,
      $R_l$ is generated by an $\ndN _0^\kappa $-homogeneous subspace
      $E_l\subset R_l$ of degree $e_l$ such that $E_l\simeq V_l$ in $\ydH $,
    \item for all $l\in L\setminus \{1,\ldots ,\kappa \}$,
      $R_l$ is generated by an $\ndN _0^\kappa $-homogeneous subspace $E_l\in
      \ydH $ of degree $\beta _l=\sum _{j=1}^\kappa n_{lj}e_j$,
      where at least two of the coefficients $n_{lj}$ are nonzero,
    \item $R\simeq \ot _{l\in L}R_l$ as $\ndN _0^\kappa $-graded objects
      in $\ydH $.
  \end{enumerate}
\end{theor}

\begin{proof}
  See \cite[Thm.\,4.12]{a-HeckGran07}. The index set $L$ corresponds to an
  appropriate subset of the set of Lyndon words in
  \cite[Thm.\,4.12]{a-HeckGran07}.  By assumption,
  the maps $\iota _u$ in \cite[Thm.\,4.12]{a-HeckGran07} can be chosen to be
  $\ndZ ^\kappa $-homogeneous morphisms in $\ydH $.
  Further, following the construction of $R_l$ as
  subquotients of $R$ and using the assumption that $R$ is
  $\ndN _0^\kappa $-graded,
  it is clear that all objects and morphisms in the theorem,
  including the direct limit $\ot _{l\in L}$,
  are $\ndN _0^\kappa $-graded.
\end{proof}

The main technical tool in this section is the following lemma.

\begin{lemma}
  Let $0\not=V\in \ydHtf $ such that $V=\oplus _{\gamma \in \ndN _0\cA
  \setminus \{0\}} V_\gamma $,
  and all finite tensor powers of $V$ are semisimple in $\ydHtf $.
  Let $d_V=\min \{d\in \ndN \,|\,V(d)\not=0\}$,
  $\kappa \in \ndN _{\ge 2}$ such that
  $V(d_V)=\oplus _{i=1}^{\kappa -1}V_i$ is a decomposition into
  irreducible objects in $\ydHtf $,
  and let $V_{\kappa }=\oplus _{n=d_V+1}^\infty V(n)$.
  Then there is a totally ordered index set $(L,\le )$
  and a family $(W_l)_{l\in L}$, $0\not=W_l\in \ydHtf $ for all $l\in L$,
  such that
  \begin{enumerate}
    \item $\{1,\ldots ,\kappa {-}1\}\subset L$ and
      $W_l\simeq V_l$ in $\ydHtf $ for $l\in \{1,\ldots ,\kappa {-}1\}$,
    \item if $l\in L\setminus \{1,2,\ldots ,\kappa -1\}$, then
      $W_l=\oplus _{n=d_V+1}^\infty W_l(n)\in \ydHtf $,
    \item $\NA (V)\simeq \ot _{l\in L}\NA (W_l)$ as objects in $\ydHtf $.
  \end{enumerate}
  \label{le:BVdec}
\end{lemma}

\begin{proof}
  Let $H'$ be as in Rem.~\ref{re:Heq}.
  We will apply Thm.~\ref{th:GH}
  to objects in $\ydHp $ instead of $\ydH $.
  
  Recall that $\NA (V)$ is a connected braided Hopf algebra
  in $\ydHtf $ generated by $V\subset P(\NA (V))$.
  Let $\{e_i\,|\,1\le i\le \kappa \}$ be the standard basis of
  $\ndZ ^{\kappa }$. The assignment
  \begin{align*}
    \udeg \,v=e_i,\qquad \text{where $v\in V_i$,}
  \end{align*}
  defines an $\ndN _0^{\kappa }$-grading on $V$, and this extends to an $\ndN
  _0^{\kappa }$-grading of $\NA (V)$ which is compatible with the $\ndZ
  ^\theta $-grading: $\ndZ ^\theta $-homogeneous components of a $\ndZ ^\kappa
  $-homogeneous element are $\ndZ ^\kappa $-homogeneous and \textit{vice
  versa}. Thus we may regard $\NA (V)$ as a $\ndZ ^\kappa $-graded object in
  $\ydHp $.
  Apply Thm.~\ref{th:GH} to $\NA (V)\in \ydHp $. This is possible, since the
  existence of the splittings $T/T'\to R$ in $\ydHp $ follows from the
  assumption that finite tensor powers of $V$ are semisimple in $\ydHp $.
  We conclude that there exists a totally ordered index set $(L,\le )$
  and a family $(R_l)_{l\in L}$ of connected $\ndZ ^{\theta +\kappa }$-graded
  braided Hopf algebras $R_l\in \ydH $, such that Thm.~\ref{th:GH}(1)--(3)
  hold.
  Since $E_l$ is $\ndN _0^\kappa $-homogeneous for all $l\in L$,
  the degrees of the $\ndN _0^\kappa $-homogeneous components of $R_l$ are of
  the form $n\,\udeg \,E_l$, where $n\in \ndN _0$.
  Further,
  Thm.~\ref{th:RisNA} implies
  (use Rem.~\ref{re:Heq} with $\kappa +\theta $ instead of $\kappa $)
  that for all $l\in L$ there exists
  an $\ndN _0^\kappa $-graded
  object $W_l\in \ydHp $ such that $E_l\subset W_l$ and
  $R_l\simeq \NA (W_l)$ in $\ydHp $.
  By the above,
  \begin{itemize}
    \item[($*$)]
      the degrees of the $\ndN _0^\kappa $-homogeneous components of
      $W_l$ are multiples of $\udeg \,E_l$.
  \end{itemize}
  Then Thm.~\ref{th:GH}(3)
  and relation $\NA (V)\in \ydHtf $
  imply that
  \begin{itemize}
    \item[($**$)]
      $\NA (V)\simeq \ot _{l\in L}\NA (W_l)$ in $\ydHtkf $ and in
      $\ydHtf $,
  \end{itemize}
  and hence Claim~(3) of the lemma holds.

  Let first $l\in \{1,2,\ldots ,\kappa -1\}$. Then
  \[ \NA (V)=\NA (V_l\oplus \tilde{V}_l)=\NA (V_l)\oplus (\tilde{V}_l),\]
  where
  $\tilde{V}_l=\oplus _{1\le j\le \kappa ,j\not=l}V_j$ and
  $(\tilde{V}_l)\subset \NA (V)$
  is the ideal generated by $\tilde{V}_l$. Thus
  $$\NA (V_l)=\oplus _{n\in \ndN _0} \{v\in \NA (V)\,|\,\udeg \,v=
  ne_l\}.$$
  Further, $\NA (W_l)$ is a connected braided Hopf algebra generated by
  $W_l$, and $V_l\simeq E_l\subset W_l$.
  By ($*$), ($**$), and since
  the $\ndN _0^{\kappa }$-homogeneous
  components of $\NA (V_l)$ are \fd{}, it follows that $V_l\simeq W_l$
  in $\ydHtkf $. This gives (1).

  Let now $l\in L\setminus \{1,2,\ldots ,\kappa -1\}$ and $W_{l\beta }$ an
  $\ndN _0^{\kappa }$-homogeneous component of $W_l$. Then either $\beta =e
  _{\kappa }$ or $\beta =\sum _{j=1}^{\kappa }n_j e_j$ with $\sum
  _{j=1}^{\kappa }n_j\ge 2$. In the first case ($**$) implies that
  $W_{l\beta }\subset V_{\kappa }$, and hence $d_{W_l}\ge d_V+1$.
  In the second case $d_{W_l}\ge 2d_V\ge d_V+1$, since $d_V\ge 1$.
  Hence Claim~(2) is proven.
\end{proof}

\begin{theor}\label{th:BVirred}
  Let $\theta \in \ndN $ and let $V_1,\ldots ,V_\theta $
  be \fd{} irreducible objects in $\ydH $. Let $V=\oplus _{i=1}^\theta V_i$,
  and assume that finite tensor powers of $V$ are semisimple in $\ydH $.
  Define a $\ndZ ^\theta $-grading of $\NA (V)$ such that $\deg V_i=\al _i$
  for all $i$.

  (1)
  There exists a totally ordered index set $(L,\le )$ and a family
  $(W_l)_{l\in L}$ of irreducible objects $W_l\in \ydHtf $ with
  $\deg W_l\in \ndN _0\cA $ for all $l$, and
  \begin{align*}
    \NA (V)\simeq \ot _{l\in L}\NA (W_l)
    \quad \text{in $\ydHtf $.}
  \end{align*}

  (2)
  If $\NA (V)\simeq \ot _{l\in L}\NA (W_l)$ and
  $\NA (V)\simeq \ot _{l'\in L'}\NA (W'_{l'})$
  for index sets $(L,\le )$, $(L',\le )$, and families
  $(W_l)_{l\in L}$, $(W'_{l'})_{l'\in L'}$ as in (1),
  then there exists a bijection $\varphi :L\to L'$ such that
  $W_l\simeq W'_{\varphi (l)}$ in $\ydHtf $
  for all $l\in L$.
\end{theor}

\begin{proof}
  The uniqueness follows from Lemma~\ref{le:undec} below.

  For the proof of the existence of the family $(W_l)_{l\in L}$ we first
  construct an inverse system of totally ordered sets and corresponding
  families of objects in $\ydHtf $. This inverse system gives rise to a direct
  system in a natural way, and the limit of this direct system will be the
  family $(W_l)_{l\in L}$ we are looking for.

  First we define recursively for all $k\in \ndN _0$
  a totally ordered index set $(L^k,\le )$ and a family
  $(W^k_l)_{l\in L^k}$, such that
  \begin{itemize}
    \item[($*$)] $0\not=W^k_l\in \ydHtf $,
      $W^k_l$ is irreducible or $W^k_l=\oplus _{n=k+1}^\infty W^k_l(n)$ for
      all $l\in L^k$, and
      $\NA (V)\simeq \ot _{l\in L^k} \NA (W^k_l)$ in $\ydHtf $.
  \end{itemize}
  Note that then for all $k,l,n$, $W^k_l(n)$ is isomorphic to a direct summand
  of $V^{\ot n}$, and hence all finite tensor powers of $W^k_l$ are semisimple
  in $\ydHtf $.

  For $k=0$ let $L^0=\{1\}$ and $W^0_1=V$.
  If $k\ge 0$, $l\in L^k$, and $W^k_l$ is not irreducible, then
  we choose a totally ordered index set $(L_{kl},\le )$ and a family
  $(W_{klm})_{m\in L_{kl}}$ of objects in $\ydHtf $
  such that
  \[\NA (W^k_l)\simeq \ot _{m\in L_{kl}} \NA (W_{klm}) \]
  and $0\not=W_{klm}\in \ydHtf $ is irreducible or
  $W_{klm}=\oplus _{n=k+2}^\infty W_{klm}(n)$. This is possible by
  Lemma~\ref{le:BVdec}.
  Now we define
  \begin{align}
    {L'}^k=&\{l\in L^k\,|\,W^k_l \text{ is irreducible}\},
    \label{eq:L'k}\\
    L^{k+1}=&{L'}^k\sqcup \{(l,m)\,|\,l\in L^k\setminus {L'}^k,m\in L_{kl}\},
    \quad \text{(disjoint union)}\notag \\
    W^{k+1}_j=&
    \begin{cases}
      W^k_j & \text{if $j\in {L'}^k$,}\\
      W_{klm} & \text{if $j=(l,m)$, $l\in L^k\setminus {L'}^k$, $m\in L_{kl}$.}
    \end{cases}\notag
  \end{align}
  Let $\le $ be the natural order on $L^{k+1}$. That is, if $l_1,l_2\in
  {L'}^k$, then $l_1\le l_2$ if and only if this relation holds in $L^k$.
  If $l_1\in {L'}^k$, $l_2=(\tilde{l},m)\in L^{k+1}\setminus {L'}^k$, then
  $l_1\le l_2$ if and only if $l_1\le \tilde{l}$ in $L^k$,
  and $l_2\le l_1$ if and only if $\tilde{l}\le l_1$ in $L^k$.
  Finally, if $l_1=(\tilde{l}_1,m_1)$,
  $l_2=(\tilde{l}_2,m_2)\in L^{k+1}\setminus {L'}^k$, then
  $l_1\le l_2$ if and only if $\tilde{l}_1<\tilde{l}_2$ in $L^k$
  or $\tilde{l}_1=\tilde{l}_2$ and $m_1\le m_2$ in $L_{k\tilde{l}_1}$.
  Clearly, $(L^{k+1},\le )$ is totally ordered, and the family
  $(W^{k+1}_l)_{l\in L^{k+1}}$ satisfies the properties in ($*$).

  For all $k\in \ndN _0$ define ${L'}^k$ as in Eq.~\eqref{eq:L'k}, and let
  $L=\cup _{k\in \ndN }{L'}^k$ with the total order $\le $ induced by the
  orders defined on the sets ${L'}^k$. Note that ${L'}^k\subset {L'}^{k+1}$
  and $W^{k+1}_l=W^k_l$ for all $k\in \ndN _0$ and $l\in {L'}^k$. Define
  $W_l=W^k_l$ for all $l\in {L'}^k\subset L$. Then $0\not= W_l\in \ydHtf $ is
  irreducible and has positive degree (with respect to the basis $\{\al
  _i\,|\,1\le i\le \theta \}$) for all $l\in L$.
  We prove that $\ot _{l\in L}\NA (W_l)\simeq \NA (V)$ in $\ydHtf $.
  Since ${L'}^k\subset L^k$ for all $k\in \ndN _0$, $\ot _{l\in {L'}^k}\NA
  (W_l)$
  is isomorphic to a subobject of $\NA (V)$ for all $k\in \ndN _0$ by ($*$).
  Moreover, ($*$) also implies that
  \begin{align}
    \NA ^n(V)\simeq 
    \big(\ot _{l\in {L'}^k}\NA (W_l) \big)(n)
    \label{eq:BVniso}
  \end{align}
  for all $k,n\in \ndN _0$ with $n\le k$. The construction of $L$ and
  $(W_l)_{l\in L}$ shows that
  \begin{align*}
    \big(\ot _{l\in {L'}^k}\NA (W_l) \big)(n)
    =\big(\ot _{l\in {L'}^{k+1}}\NA (W_l) \big)(n)
    =\big(\ot _{l\in L}\NA (W_l) \big)(n)
  \end{align*}
  for all $k,n\in \ndN _0$ and $n\le k$. Hence
  \begin{align*}
    \NA (V)=& \oplus _{n=0}^\infty \NA ^n(V)
    \simeq \oplus _{k=0}^\infty  
    \big(\ot _{l\in {L'}^k}\NA (W_l) \big)(k)\\
    =& \oplus _{k=0}^\infty 
    \big(\ot _{l\in L}\NA (W_l) \big)(k)
    =  \ot _{l\in L}\NA (W_l).
  \end{align*}
  This finishes the proof of the theorem.
\end{proof}

\begin{remar}\label{re:Khar}
  In the setting of Thm.~\ref{th:BVirred},
  assume that $H$ is the group algebra of an abelian group, and $V_i$ is
  one-dimensional for all $i\in \{1,2,\dots ,\theta \}$. In \cite{a-Khar99},
  Kharchenko gave a construction of a PBW-basis for a class of Hopf algebras,
  which can be applied to $\NA (V)$. Thm.~\ref{th:BVirred} can be viewed
  as a generalization of Kharchenko's basis
  to Nichols algebras over arbitrary Hopf algebras in the sense that the PBW
  generators of Kharchenko are replaced by the Yetter-Drinfeld modules $W_l$.
\end{remar}

\begin{lemma}
  Let $\theta \in \ndN $, $(L,\le )$, $(L',\le )$ totally ordered index sets,
  and
  $(W_l)_{l\in L}$, $(W'_{l'})_{l'\in L'}$ families of irreducible
  objects in $\ydHtf $.  Assume that
  $\deg W_l\in \sum _{i=1}^\theta \ndN _0\al _i \setminus \{0\}$ for all $l\in
  L$. If $\ot _{l\in L}\NA (W_l)\simeq \ot _{l'\in L'}\NA (W'_{l'})$,
  and all $\ndZ ^\theta $-homogeneous components of
  $\ot _{l\in L}\NA (W_l)$ are \fd{},
  then there exists a bijection $\varphi :L\to L'$ such that
  $W_l\simeq W'_{\varphi (l)}$ in $\ydHtf $
  for all $l\in L$.
  \label{le:undec}
\end{lemma}

\begin{proof}
  Let $l'\in L'$.
  Since $W'_{l'}$ is a direct summand of $\ot _{l''\in L'}\NA (W'_{l''})
  \simeq
  \ot _{l\in L}\NA (W_l)$ and the degree $0$ component of
  $\ot _{l\in L}\NA (W_l)$ is $\fie $, we obtain that
  $\deg W'_{l'}\in \sum _{i=1}^\theta \ndN _0\al _i \setminus \{0\}$.

  For $k\in \ndN _0$ let $L_k=\{l\in L\,\big|\,|\deg W_l|=k\}$.
  Then $L_0=\emptyset $, $L=\cup _{k\in \ndN }L_k$, and
  $L_k$ is a finite set for all $k\in \ndN $, since the $\ndZ
  ^\theta$-homogeneous components of
  $\ot _{l\in L}\NA (W_l)$ are \fd{}.
  Similarly $L'=\cup _{k\in \ndN }L'_k$. It suffices to show that
\begin{itemize}
  \item[($*$)] 
  for all $k\in \ndN _0$ there is a bijection $\varphi _k:L_k\to L'_k$ such
  that $W_l\simeq W'_{\varphi _k(l)}$ in
  $\ydHtf $ for all $k\in \ndN _0$, $l\in L_k$.
\end{itemize}
  This is clear for $k=0$, since
  $L_0=L'_0=\emptyset $.

  %Let $\gamma \in \ndZ ^\theta $ with $|\gamma |=k$. First,
  Let $k\in \ndN $. Then
  \begin{align*}
  %\NA ^0(V)=&\fie
  %=\big(\ot _{l\in L}\NA (W_l)\big)(0)
  %=\big(\ot _{l'\in L'}\NA (W'_{l'})\big)(0),\\
  \big(\ot _{l\in L}\NA (W_l)\big)(k)\simeq &
  \oplus _{l\in L_k}W_l\oplus \big(\otimes _{l\in L_1\cup L_2\cup \cdots \cup
  L_{k-1}}\NA (W_l)\big)(k),\\
  \big(\ot _{l'\in L'}\NA (W'_{l'})\big)(k)\simeq &
  \oplus _{l'\in L'_k}W'_{l'}\oplus
  \big(\otimes _{l'\in L'_1\cup L'_2\cup \cdots \cup
  L'_{k-1}}\NA (W'_{l'})\big)(k)
  \end{align*}
  in $\ydHtf $.
  Hence ($*$) follows by induction on $k$ and by Krull-Remak-Schmidt.
%  follow
%  This implies the 
%  and the last line of the above equation implies that
%  $\oplus _{l\in L_1}W_l=\oplus _{l'\in L'_1}W'_{l'}=V$.
%  Since $W_l$ and $W'_{l'}$ are irreducible
%  for all $l\in L$, $l'\in L'$, for each $i\in \{1,2,\ldots ,\theta \}$ there
%  exist unique $l(i)\in L$, $l'(i)\in L'$, such that $W_{l(i)}\simeq V_i$,
%  $W'_{l'(i)}\simeq V_i$. Further, if $l\in L$, $l\not=l(i)$ for all $i\in
%  \{1,2,\ldots ,\theta \}$,
%  then $\deg W_l\ge 2$, and if $l'\in L'$, $l'\not=l'(i)$ for all $i$,
%  then $\deg W'_{l'}\ge 2$.
%  By induction on $k\in \ndN $ one concludes from the equations
%  \begin{align*}
%    \big(\ot _{l\in L}\NA (W_l)\big)(k)
%    =\big(\ot _{l'\in L'}\NA (W'_{l'})\big)(k),
%  \end{align*}
%  that for all $k\in \ndN $ there exists a (not necessarily unique) bijection
%  $\varphi _k:L_k\to L'_k$ such that
%  $W_l\simeq W'_{\varphi _k(l)}$ for all $l\in L_k$. Moreover, if $l\notin
%  \cup _{n=1}^k L_n$,
%  then $\deg W_l>k$, and if $l'\notin \cup _{n=1}^k L'_n$, then $\deg
%  W'_{l'}>k$.
%  Since all objects $W_l$, $W'_{l'}$ are homogeneous of finite degree, we
%  conclude the uniqueness of the family $(W_l)_{l\in L}$ up to the ordering
%  of $L$.
\end{proof}

The following similar results will be used later on.

\begin{lemma}
  Let $K,K'$ and $B$ be $\ndZ ^\theta $-graded objects in $\ydH $ such that
  all homogeneous components of $K$ and $B$ are finite-dimensional,
  $B_0$ is isomorphic to the trivial object $\fie \in \ydH $,
  and $K_\gamma =B_\gamma =0$ for all $\gamma \in \ndZ
  ^\theta \setminus \sum _{i=1}^\theta \ndN _0\al _i$.
  If $K\ot B\simeq K'\ot B$ in $\ydHtf $,
  then $K\simeq K'$ in $\ydHtf $.
  \label{le:KBiso}
\end{lemma}

\begin{proof}
  Since $B_0\simeq \fie $, $K'_\gamma $ is (isomorphic to) a direct summand of
  $(K'\ot B)_\gamma $ for all $\gamma \in \ndZ ^\theta $. Hence the
  isomorphism $K\ot B\simeq K'\ot B$ implies that $K'_\gamma =0$ for all
  $\gamma \in \ndZ ^\theta \setminus \sum _{i=1}^\theta \ndN _0\al _i$.

  Similarly to the proof of the previous lemma,
  \begin{align*}
    (K\ot B)(k)\simeq & K(k)\oplus \oplus _{1\le k'\le k}
    K(k-k')\ot B(k'),\\
    (K'\ot B)(k)\simeq & K'(k)\oplus \oplus _{1\le k'\le k}
    K'(k-k')\ot B(k')
  \end{align*}
  in $\ydHtf $ for all $k\in \ndN _0$.
  Since the $\ndN _0$-homogeneous components of $K\ot B$ are \fd{},
  the claim follows by induction on $k$ and by Krull-Remak-Schmidt.
\end{proof}

\begin{lemma}
  Let $\theta \in \ndN $ and let $V_1,\ldots ,V_\theta $
  be \fd{} irreducible objects in $\ydH $. Let $V=\oplus _{i=1}^\theta V_i$,
  and define a $\ndZ ^\theta $-grading of $\NA (V)$
  such that $\deg V_i=\al _i$ for all $i$.
  Assume that
  there exists a totally ordered index set $(L,\le )$ and a family
  $(W_l)_{l\in L}$ of irreducible objects
  $W_l\in \ydHtf $ such that $\deg W_l\in \ndN _0\cA $ for
  all $l$, and
  \begin{align*}
    \NA (V)\simeq \ot _{l\in L}\NA (W_l)
  \end{align*}
  in $\ydHtf $.
  Then for each $i\in \{1,\ldots ,\theta \}$ there exists a unique $l(i)\in L$
  such that $\deg W_{l(i)}\in \ndN \al _i$. Moreover,
  $\deg W_{l(i)}=\al _i$ and $W_{l(i)}\simeq V_i$.
  \label{le:BVdec1}
\end{lemma}

\begin{proof}
  Let $i\in \{1,\ldots ,\theta \}$.
  The subspace $\oplus _{n\in \ndZ }\NA (V)_{n\al _i}\subset \NA (V)$
  is the subalgebra $\NA (V_i)$, and $\NA (V)_{\al _i}=V_i$.
  Since $\deg W_l\in \sum _{j=1}^\theta \ndN _0\al _j\setminus 0$, we obtain
  that
  \begin{align}
    \NA (V_i)\simeq
    \oplus _{n\in \ndZ }\big(\ot _{l\in L}\NA (W_l)\big )_{n\al _i}=
    \ot _{l\in L_i}\NA (W_l),
    \label{eq:BViso}
  \end{align}
  where $L_i=\{l\in L\,|\,\deg W_l\in \ndN \al _i\}$. Hence
  \[V_i=\NA (V_i)_{\al _i}\simeq 
  \big(\ot _{l\in L}\NA (W_l)\big)_{\al _i}=\oplus _{l\in L_i,\deg W_l
  =\al _i}W_l,\]
  and since $V_i$ is irreducible, it follows that there is a unique $l=l(i)\in
  L$ such that $V_i\simeq W_l$ and $\deg W_l=\al _i$. Then
  Eq.~\eqref{eq:BViso} implies that $L_i=\{l(i)\}$ which proves the lemma.
\end{proof}

\section{Weyl groupoids and root systems}
\label{sec:Wgrs}

In this section we give the definition of generalized root systems in
\cite{a-HeckYam08} in terms of category theory, see \cite{p-CH08},
and recall some results which
will be needed in the following sections.

Let $I$ be a non-empty finite set and $(\al _i)_{i\in I}$
the standard basis of $\ndZ ^I$.

Recall from \cite[\S 1.1]{b-Kac90}
that a generalized Cartan matrix $\Cm =(\cm _{ij})_{i,j\in I}$
is a matrix in $\ndZ ^{I\times I}$ such that
  \begin{enumerate}
    \item[(M1)] $\cm _{ii}=2$ and $\cm _{jk}\le 0$ for all $i,j,k\in I$ with
      $j\not=k$,
    \item[(M2)] if $i,j\in I$ and $\cm _{ij}=0$, then $\cm _{ji}=0$.
  \end{enumerate}

Let $\cM $ be a non-empty set, and for all $i\in I$ and $N\in \cM $ let
$\rfl _i : \cM \to \cM $ be a map and $\Cm ^N=(\cm ^N_{jk})_{j,k \in I}$
a generalized Cartan matrix.
The quadruple
\[\cC = \cC (I,\cM ,(\rfl _i)_{i \in I},
(\Cm ^N)_{N \in \cM }),\]
is called a \textit{Cartan scheme} if
\begin{enumerate}
\item[(C1)] $\rfl _i^2 = \id$ for all $i \in I$,
\item[(C2)] $\cm ^N_{ij} = \cm ^{\rfl _i(N)}_{ij}$ for all $N \in \cM $
  and $i,j\in I$.
\end{enumerate}

Let $\cC = \cC (I,\cM ,(\rfl _i)_{i \in I}, (\Cm ^N)_{N \in \cM })$ be a
Cartan scheme. For all $i \in I$, $N\in \cM $ define $s_i^N \in
\Aut(\ndZ ^I)$ by
\[s_i^N(\al _j) = \al _j - \cm _{ij}^N \al _i \qquad \text{ for all } j \in
I.\]

Recall that a groupoid is a category where all morphisms are isomorphisms.
The {\em Weyl groupoid of $\cC $}  is the groupoid
$\Wg (\cC )$ with $\Ob (\Wg (\cC ))=\cM $, where the morphisms are 
generated by all $s_i^N : N \to \rfl _i(N)$ with $i \in I$, $N \in \cM $.
Formally, for $N,P\in \cM $ the set $\Hom (N,P)$ consists of the triples
$(P,s,N)$, such that
\[
s=s_{i_n}^{\rfl _{i_{n-1}}\cdots \rfl _{i_1}(N)}\cdots
s_{i_2}^{\rfl _{i_1}(N)}s_{i_1}^N
\]
and $\rfl _{i_n}\cdots \rfl _{i_2}\rfl _{i_1}(N)=P$
for some $n\in \ndN _0$ and $i_1,\ldots ,i_n\in I$.
The composition of morphisms is induced by the group structure of
$\Aut (\ndZ ^I)$:
\[ (Q,g,P)\circ (P,f,N) = (Q,gf, N)\]
for all $(Q,g,P),(P,f,N)\in \Hom (\Wg (\cC ))$. If $w=(P,f,N)\in \Hom (\Wg
(\cC ))$ and $\al \in \ndZ ^\theta $, then we define $w(\al )=f(\al )$.

Note that the inverse of the generating morphism
$(\rfl _i(N),s_i^N,N)$ %\in \Hom (N,\rfl _i(N))$
is $(N,s_i^{\rfl _i(N)},\rfl _i(N))$, since $s_i^N$ is a reflection, and
$s_i^N=s_i^{\rfl _i(N)}$ and $\rfl _i^2(N)=N$ by definition.

We say that
\[\Rwg = \Rwg (\cC , (\rsys (N))_{N \in \cM })\]
is a {\em root system of type} $\cC $
if $\cC = \cC (I,\cM ,(\rfl _i)_{i \in I}, (\Cm ^N)_{N \in \cM })$
is a Cartan scheme and $\rsys (N) \subset \ndZ ^I$, where $N \in \cM $,
are subsets such that
\begin{enumerate}
\item[(R1)] $\rsys (N)=(\rsys (N)\cap \ndN _0^I)\cup -(\rsys (N)\cap \ndN _0^I)$
  for all $N \in \cM $,
\item[(R2)] $\rsys (N)\cap \ndZ \al _i=\{\al _i,-\al _i\}$ for all $i\in I$,
  $N\in \cM $,
\item[(R3)] $s_i^N(\rsys (N))= \rsys (\rfl _i(N))$ for all $i \in I$, $N\in \cM $,
\item[(R4)] $(\rfl _i \rfl _j)^{m_{i,j;N}}(N) = N$ for all $i,j \in I$ and
  $N\in \cM $ such that $i\not=j$ and
  $m_{i,j;N}:= |\rsys (N)\cap (\ndN _0\al _i + \ndN _0 \al_j)|$ is finite.
\end{enumerate}
We note that Axiom~(M2) is redundant for root systems by
\cite[Lemma\,2.5]{p-CH08}.

If $\Rwg (\cC , (\rsys (N))_{N \in \cM })$ is a root system of type $\cC $,
then $\Wg (\Rwg ) := \Wg (\cC )$ is called the {\em
Weyl groupoid of} $\Rwg $.
The elements of $\rsys _+(N):=\rsys (N)\cap \ndN _0^I$ and
$\rsys _-(N):=-\rsys _+(N)$ are called \textit{positive} and
\textit{negative roots}, respectively. Following \cite[\S 5.1]{b-Kac90}
we say that
the roots $w(\al _i)\in \rsys (P)$, where $w\in \Hom (N,P)$, $N,P\in \cM $,
and $i\in I$, are \textit{real roots}.
The set of real roots and positive real roots is denoted by $\rsys \re (N)$
and $\rsys _+\re (N)$, respectively.
Note that (R3) implies that $w(\rsys (N))=\rsys (P)$
for all $N,P\in \Ob (\Wg (\Rwg ))$ and $w\in \Hom (N,P)$.

Recall that a groupoid $G$ is \textit{connected},
if for all $N,P\in \Ob (G)$
the set $\Hom (N,P)$ is non-empty. It is \textit{finite},
if $\Hom (G)$ is finite.

\begin{lemma}\cite[Lemma\,2.11]{p-CH08}\label{le:finrs}
  Let $\cC $ be a Cartan scheme and $\Rwg $ a root system of type $\cC $.
  Assume that $\Wg (\Rwg )$ is connected.
  Then the following are equivalent.
  \begin{enumerate}
    \item $\rsys (N)$ is finite for all
      $N\in \Ob (\Wg (\Rwg ))$.
    \item $\rsys (N)$ is finite for at least one
      $N\in \Ob (\Wg (\Rwg ))$.
    \item $\rsys \re (N)$ is finite for all
      $N\in \Ob (\Wg (\Rwg ))$.
    \item $\Wg (\Rwg )$ is finite.
  \end{enumerate}
  \label{le:finHom}
\end{lemma}

\begin{propo} \cite[Prop.\,2.12]{p-CH08} \label{pr:allposroots}
  Let $\cC $ be a Cartan scheme and $\Rwg $ a root system of type $\cC $.
  If $\Wg (\Rwg )$ is
  %connected and 
  finite, then all roots are real.
\end{propo}

\begin{lemma} \cite[Lemma\,8(iii)]{a-HeckYam08}\label{le:finrs2}
  Let $\cC $ be a Cartan scheme and $\Rwg $ a root system of type $\cC $.
  Let $N,P\in \Ob (\Wg (\Rwg ))$ and $w\in \Hom (N,P)\subset
  \Hom (\Wg (\Rwg ))$ such that $w(\rsys _+(N))\subset \rsys _-(P)$. Then
  $\rsys (N)$ is finite.
\end{lemma}

In general the Cartan matrices $A^N$ are not of finite type if the Weyl
groupoid $\Wg (\Rwg )$ is finite \cite[Prop.\,5.1(2)]{p-CH08}. 
But in the special case, when the Cartan scheme $\cC $ is {\em standard},
i.\,e. $\Cm ^N = \Cm^M$ for all $N,M \in \Ob(\Wg (\cC ))$, the following
holds (see \cite[Thm.\,3.3]{p-CH08} for a slightly different proof).

\begin{corol}\label{co:standard}
Let $\cC $ be a standard Cartan scheme with generalized Cartan matrix $\Cm
=\Cm ^N$ for all $N \in \Ob(\Wg (\cC ))$. Let $\Rwg $ be a root system of type
$\cC $.
Then the following are equivalent.
\begin{enumerate}
\item $\Wg (\Rwg )$ is finite.
\item $\Cm $ is a Cartan matrix of finite type.
\end{enumerate}
\end{corol}

\begin{proof}
Let $W(\Cm )$ be the Weyl group of the Kac-Moody Lie
algebra of $\Cm $.
In particular, $W(\Cm )$ is isomorphic to the subgroup of $\Aut (\ndZ ^\theta )$
generated by the reflections $s_i\in \Aut (\ndZ ^\theta )$, where
\begin{align*}
  s_i(\al _j)=\al _j-\cm _{ij}\al _i\qquad \text{for all $i,j\in I$}.
\end{align*}

(1)$\Rightarrow$(2):
Since $\cC $ is standard, the map $\Hom(\Wg (\Rwg )) \to
W(\Cm )$, $(P,s,N) \mapsto s$, is well-defined and surjective. Hence $W(\Cm )$
is finite by (1), and $\Cm $ is of finite type.

(2)$\Rightarrow$(1): Since $\Cm $ is of finite type, the Weyl group $W(\Cm )$
acts transitively on bases of the root system. Hence there exists a
permutation $\tau $ of $I$ and an element $w \in W(\Cm )$ such that $w(\al _i) =
-\al _{\tau (i)}$ for all $i\in I$. Let $n\in \ndN $ and $i_1,\ldots ,i_n\in
I$ such that $w=s_{i_n}\cdots s_{i_2}s_{i_1}$.
Then $w=s_{i_n}^{\rfl _{i_{n-1}}\cdots \rfl _{i_1}(N)}
\cdots s_{i_2}^{\rfl _{i_1}(N)}s_{i_1}^N$ for all $N\in \cM $,
since $\cC $ is standard.
Hence $w(\rsys _+(N))\subset
\rsys _-(\rfl _{i_n}\cdots \rfl _{i_2}\rfl _{i_1}(N))$ for all $N\in \cM $.
This proves (1) by Lemmata~\ref{le:finrs} and \ref{le:finrs2}.
\end{proof}

\begin{remar}
  Let $\cC =\cC (I,\cM ,(\rfl _i)_{i\in I},(\Cm ^N)_{N\in \cM })$ be
  a standard Cartan scheme and $\Rwg =\Rwg (\cC ,(\rsys ^N)_{N\in \cM })$
  a root system of type $\cC $. Define $\Cm =\Cm ^N$ for all $N\in \cM $,
  and let $W(A)$ denote the Weyl group of the Kac-Moody Lie algebra of $\Cm $.
  Then the maps $s _i^N\in \Aut (\ndZ ^\theta )$,
  where $i\in I$ and $N\in \cM $, are independent of $N$.
  Assume that $\rsys (N)=\rsys \re (N)$ for all $N\in \cM $. Then
  $\rsys (N)=\{w(\al _i)\,|\,w\in W(\Cm ),i\in I\}$ for all $N\in \cM $,
  and hence it is independent of $N$. In this case $\Rwg $ is equivalent to an
  action of $W(\Cm )$ on $\cM $.
\end{remar}

There is an important necessary condition for
the finiteness of the Weyl groupoid of a
Cartan scheme. For the proof we need a lemma,
which is a special case of \cite[Lemma\,9]{a-Heck08a}.

\begin{lemma}
  Let $T\subset \mathrm{SL}(2,\ndZ )$ be a non-empty
  subsemigroup generated by matrices of the form
  $\begin{pmatrix}
    a & -b \\ c & -d
  \end{pmatrix}$ with $0<d<b<a$. Then all elements of $T$ are of this form.
  In particular, $\id \notin T$.
  \label{le:infsemigrp}
\end{lemma}

\begin{propo}\label{pr:aij=-1}
  Let $\cC =\cC (I,\cM ,(\rfl _i)_{i\in I},(\Cm ^N)_{N\in \cM })$
  be a Cartan scheme of rank two.
  If $\Hom (\Wg (\cC ))$ is a finite set, then there exist $N\in \cM $
  and $i,j\in I$ such that $\cm ^N_{ij}\in \{0,-1\}$.
\end{propo}

\begin{proof}
  For all $a\in \ndZ $ let
  \[ \eta _1(a)=
  \begin{pmatrix}
    -1 & a \\ 0 & 1
  \end{pmatrix}
  ,\quad \eta _2(a)=
  \begin{pmatrix}
    1 & 0 \\ a & -1
  \end{pmatrix}. \]
  If $N\in \cM _2(M)$, then
  \begin{equation}
  \begin{aligned}
    &s_1^{\rfl _2\cdots \rfl _1\rfl _2(M)}\cdots 
    s_2^{\rfl _1\rfl _2(M)}s_1^{\rfl _2(M)}s_2^M\\
    &\quad =\eta _1(-a_{12}^{\rfl _2\cdots \rfl _1\rfl _2(M)})\cdots 
    \eta _2(-a_{21}^{\rfl _1\rfl _2(M)})\eta _1(-a_{12}^{\rfl _2(M)})
    \eta _2(-a_{21}^M).
  \end{aligned}
    \label{eq:sss}
  \end{equation}
  Assume that $a^N_{12},a^N_{21}<-1$ for all $N\in \cM _2(M)$.
  Then the matrices
  \begin{align*}
    \eta _1(a')\eta _2(a'')=
    \begin{pmatrix}
      a'a''-1 & -a'\\
      a'' & -1
    \end{pmatrix},
  \end{align*}
  where $a'=-a_{12}^{\rfl _2(N)}$, $a''=-a_{21}^N$ for some $N\in \cM _2(M)$,
  satisfy the assumption of Lemma~\ref{le:infsemigrp}.
  Lemma~\ref{le:infsemigrp} implies that the set of products given
  in Eq.~\eqref{eq:sss} is infinite. This is a contradiction to the finiteness
  of $\Hom (\Wg (\cC ))$.
\end{proof}

\section{Weyl groupoids and root systems for Nichols algebras}

Let $\icyH $ denote the set of isomorphism classes of \fd{} irreducible
Yetter-Drinfeld modules over $H$. For any $X\in \ydH $ let $[X]$
denote the isomorphism class of $X$. Let $\theta \in \ndN $, $\Ib
=\{1,2,\ldots ,\theta \}$, and
$\cM _\theta =\big(\icyH \big) ^\theta $. The standard basis of $\ndZ ^\theta
$ is denoted by
$\{\al _1,\ldots ,\al _\theta \}$.

\begin{defin}
  Let $N\in \cM _\theta $. Choose $N_1,\ldots ,N_\theta \in \ydH $ such that
  $N=([N_1],\ldots ,[N_\theta ])$. 
  Define
  $V(N)=N_1\oplus\ldots \oplus N_\theta $.
  %and
  %assume that all finite tensor powers of $V(N)$ are semisimple in $\ydH $.
  %By Thm.~\ref{th:BVirred}(1), we choose an index set $(L,\le)$ and a family 
  %$(W_l)_{l\in L}$ such that $\NA (V(N))\simeq \ot _{l\in L}\NA (W_l)$.
  Assume that there exists a totally ordered index set $(L,\le)$ and a family 
  $(W_l)_{l\in L}$ of \fd{} irreducible $\ndN _0^\theta $-graded objects in
  $\ydH $
  such that $\NA (V(N))\simeq \ot _{l\in L}\NA (W_l)$ as $\ndN _0^\theta
  $-graded objects in $\ydH $, where $\deg N_i=\al _i$ for all $i\in \Ib $.
  Define
  \begin{align}
    \rsyst _+(N)=&\{(C, \al )\,|\, C\in \icyH ,
    \al \in \ndN _0^\theta \setminus 0, \label{eq:proots}\\
    &\phantom{ \{(C, \al )\,|\, }\exists l\in L : C=[W_l],\al =\deg
    W_l\},\notag \\
    \rsyst _-(N)=&\{(C^*,-\al )\,|\,(C,\al )\in \rsyst _+(N)\},
    \label{eq:nroots}\\
    \rsyst (N)=&\rsyst _+(N)\cup \rsyst _-(N),
  \end{align}
  and for all $(C,\al )\in \rsyst _+(N)$ let
  \begin{align}
    \mult _N(C,\al )=\mult _N(C^*,-\al )=
    |\{l\in L\,|\,C=[W_l],\al =\deg W_l\}|.
    \label{eq:mult}
  \end{align}
  The sets
\begin{gather}
  \rsys _+(N)=\{\deg W_l\,|\,l\in L\},\quad 
  \rsys _-(N)=-\rsys _+(N),\\
  \rsys (N)=\rsys _+(N)\cup \rsys _-(N)
  \label{eq:rsys}
\end{gather}
  are called the \textit{set of positive roots of} $N$,
  the \textit{set of negative roots of} $N$,
  and the \textit{set of roots of} $N$, respectively.
  For any $N\in \cM _\theta $ we say that $\rsyst (N)$ \textit{is defined}, if
  the assumptions of Def.~\ref{de:rsysN} are fulfilled for $N$.
  \label{de:rsysN}
\end{defin}

\begin{remar}
  In \cite{a-Heck06a}, the root system of a Nichols algebra of diagonal type
  was defined using Kharchenko's PBW-basis. Viewing the Yetter-Drinfeld
  modules $W_l$ as generalized PBW generators, see Rem.~\ref{re:Khar},
  the definition of $\rsys (N)$
  in Eq.~\eqref{eq:rsys} generalizes the root system given in \cite{a-Heck06a}. 
\end{remar}

Let $N\in \cM _\theta $. By Thm.~\ref{th:BVirred}(1), $\rsyst (N)$ is defined,
if all tensor powers of $V(N)$ are semisimple.

By Lemma~\ref{le:undec}, the definitions of $\rsyst _\pm (N)$,
$\rsyst (N)$, $\rsys _\pm (N)$, and $\rsys (N)$
only depend on $N$, but not on the choice of
$N_1,\ldots ,N_\theta $, $(L,\le )$, and the family $(W_l)_{l\in L}$.

\begin{lemma}
  Let $N,P\in \cM _\theta $ such that $\rsyst (N)$ and $\rsyst (P)$ are
  defined. If $\rsyst (N)=\rsyst (P)$ then $N=P$.
  \label{le:NrsysN}
\end{lemma}

\begin{proof}
  See Lemma~\ref{le:undec}.
\end{proof}

\begin{defin}\label{de:cm}
  Let $N=([N_1],\ldots ,[N_\theta ])\in \cM _\theta $, $V(N)=N_1\oplus
  \cdots \oplus N_\theta $, and $i\in \Ib $. We say that
  $N$ is $i$-\textit{finite}, if
  for all $j\in \Ib \setminus \{i\}$
  $(\ad _c\,N_i)^h(N_j)=0$ for some $h\in \ndN $,
  and $\Ib $-\textit{finite}, if $N$ is $l$-finite for all $l\in \Ib $.
%\end{defin}
%
%\begin{defin}\label{de:cm}
%  Let $N=([N_1],\ldots ,[N_\theta ])\in \cM _\theta $, $V(N)=N_1\oplus
%  \cdots \oplus N_\theta $, and $i\in \Ib $.

  If $N$ is not $i$-finite, then let $r_i(N)=N$.

  Assume now that $N$ is $i$-finite. Let
  $\Cm ^N=(\cm ^N_{ij})_{i,j\in \Ib }$, where
  \begin{align}
    -\cm ^N_{i j}=&\sup \{h\in \ndN _0\,|\, (\ad _c\,N_i)^h(N_j)\not=0
    \text{ in $\NA (V(N))$}\}
    \label{eq:cm}
  \end{align}
  for $j\in \Ib \setminus \{i\}$, and $\cm ^N_{ii}=2$. Let
  $\rfl _i(N)=([N'_1],\dots ,[N'_\theta ])\in \cM _\theta $,
  where $N'_i=N_i^*$ and $N'_j=(\ad _c N_i)^{-\cm ^N_{ij}}(N_j)$ for all $j\in
  \Ib \setminus \{i\}$.
  Define $s_i^N\in \Aut (\ndZ ^\theta )$ by
  \begin{align}
    s_i^N(\al _j)=&\al _j-\cm ^N_{ij}\al _i& &
    \text{for all $j\in \Ib $,}\\
    s_i^N(C,\gamma )=&(C,s_i^N(\gamma ))& &
    \text{for all $C\in \icyH $, $\gamma \in \ndZ ^\theta $.}
    \label{eq:siNC}
  \end{align}
%  \end{align}
%If $N=([N_1],\ldots ,[N_\theta ])\in \cM _\theta $, $V(N)=N_1\oplus
%\cdots \oplus N_\theta $, and $i\in \Ib $, then
%For any $N\in \cM _\theta $ and $i\in \Ib $ we
%say that \textit{$\rfl _i(N)$ is defined} if
%$(\ad \, \NA (N_i))(N_j)$ is \fd{} for all $j\in \Ib \setminus \{i\}$.
%the assumptions of Def.~\ref{de:cm} are satisfied for $i$ and $N$.
%For $N\in \cM _\theta $ and $i\in \Ib $ such that $\rfl _i(N)$ is defined,
%we let
%\begin{align}
%
\end{defin}

By definition, the map $s_i^N$ is a reflection in the sense of
\cite[Ch.\,V,\S 2.2]{b-BourLie4-6}, since $a^N_{ii}=2$.
Note that the definition of $\rfl _i$ is slightly different from the
definition of $\mathcal{R}_i$ in
\cite[Eq.\,(3.16)]{p-AHS08}.

The next proposition shows how to compute the Yetter-Drinfeld module
$(\ad _c V)^n(W)$, where
$n\in \ndN $ and $V,W\in \ydH $. Recall the definition of $S_n$ from
Eq.~\eqref{eq:Sn}.

\begin{propo}\label{pr:Tn}
  Let $n\in \ndN $ and $V,W\in \ydH $. Then the image of the linear map
  $(S_n\ot \id )T_n\in \End \,(V^{\ot n}\ot W)$, where
  \begin{align*}
    T_n=&(\id -c_{n,n+1}^2c_{n-1,n}\cdots c_{12})\cdots 
    (\id -c_{n,n+1}^2c_{n-1,n})(\id -c_{n,n+1}^2),
  \end{align*}
  is isomorphic to $(\ad _c V)^n(W)\subset \NA (V\oplus W)$ in $\ydH $.
\end{propo}

\begin{proof}
  The kernel of $S_{n+1}\in \End ((V\oplus W)^{\ot n+1})$
  is $\NI _{T(V\oplus W)}(n+1)$, see
  Sect.~\ref{sec:prelims}.
  Hence the image of $S_{n+1}$
  is isomorphic to $\NA (V\oplus W)(n+1)$ in $\ydH $.
  The lemma follows from
  \begin{align}
    (S_n\ot \id )T_n(v_1\ot \cdots \ot v_n\ot w)
    =S_{n+1}( (\ad _c v_1)\cdots (\ad _c v_n)(w))
  \end{align}
  for all $v_1,\dots,v_n\in V$, $w\in W$, where
  $$\ad _c v: (V\oplus W)^{\ot k} \to (V\oplus W)^{\ot k+1}, \quad
  x\mapsto v\ot x-(v\_{-1}\cdot x)\ot v\_0$$
  for all $k\in \ndN _0$, $v\in V$, and $x\in (V\oplus W)^{\ot k}$.
  By Eq.~\eqref{eq:Sn} it suffices to show that
  \begin{align*}
    T_n(v_1\ot \cdots \ot v_n\ot w)=
    S_{n,1}((\ad _c v_1)\cdots (\ad _c v_n)(w))
  \end{align*}
  for all $v_1,\dots,v_n\in V$, $w\in W$. This follows by induction on $n$
  from the braid equation for $c$.
\end{proof}

  Let $N\in \cM _\theta $ and $i\in \Ib $.
  Let $\cK _i^N=\NA (V(N))^{\co \,\NA (N_i)}$
  be the algebra of right coinvariant
  elements with respect to the coaction $(\id \ot \pi )\copr _{\NA (V(N))}$,
  where
  $\pi :\NA (V(N))\to \NA (N_i)$ is the canonical projection, see
  \cite[Sect.\,3.2]{p-AHS08}.
  Further, $\NA (V(N))$, $\cK _i^N$, and $\NA (N_i)$ are $\ndZ^\theta $-graded
  objects in $\ydH $ with $\deg N_j=\al _j$ for all $j\in \Ib $.
  Then
  \begin{align}
    \NA (V(N))\simeq \cK _i^N\ot \NA (N_i)
    \label{eq:KBV}
  \end{align}
  as $\ndN _0^\theta $-graded objects in $\ydH $, see
  \cite[Lemma\,3.2(ii)]{p-AHS08} and the discussion in
  \cite[Sect.\,3.4]{p-AHS08}.

\begin{lemma}
  Let $N\in \cM _\theta $ and $i\in \Ib $. Assume that $\rsyst (N)$
  is defined.
  
  (1)
  Let $L$ and $(W_l)_{l\in L}$ as in Def.~\ref{de:rsysN}.
  Then $\cK _i^N\simeq \ot _{l\in L,l\not=l(i)} \NA (W_l)$
  as $\ndN _0^\theta $-graded objects in $\ydH $.

  (2)
  Let $i,j \in \Ib $, $i \neq j$, and $n\in \ndN _0$ such that
  $(\ad _c N_i)^n(N_j)\not=0$. Then
  $\al _j+m\al _i\in \rsys (N)$ for all $0 \leq m \leq n$.
  \label{le:Kroots}
\end{lemma}

\begin{proof}
  (1)
  Since $\rsyst (N)$ is defined,
  \[\NA (V(N))\simeq \ot _{l\in L,l<l(i)}\NA (W_l)\ot \NA (N_i)\ot \ot _{l\in
  L,l>l(i)}\NA (W_l)
  \]
  by Lemma~\ref{le:BVdec1}.
  Since the category $\ydH $ is braided,
  \begin{align}
    \NA (V(N))\simeq \ot _{l\in L,l\not=l(i)}\NA (W_l)\ot \NA (N_i).
    \label{eq:BVdec}
  \end{align}
  Thus Eqs.~\eqref{eq:BVdec}, \eqref{eq:KBV} and Lemma~\ref{le:KBiso}
  imply Claim (1).

  Now we prove (2).
  By \cite[Prop.\,3.6]{p-AHS08} the algebra $\cK_i^N$ is generated by the
  $\ndN _0^\theta $-homogeneous subspaces $(\ad _c N_i)^m(N_p)$ of degree $\al
  _p+m\al _i$, where $p\in \Ib \setminus \{i\}$ and $m\ge 0$, and hence
  \begin{equation}
    (\cK _i^N)_{\al _j+m\al _i}=(\ad _c N_i)^m(N_j).
    \label{eq:KiMdeg}
  \end{equation}
  On the other hand, since $\deg W_l\in (\sum _{j=1}^\theta \ndN _0\al _j)
  \setminus \ndN _0\al _i$ for all $l\in L\setminus \{l(i)\}$ by
  Lemma~\ref{le:BVdec1}, and all tensor factors are $\ndN _0^\theta
  $-graded, we obtain that
  \begin{equation}
    \Big(\ot _{l\in L,l\not=l(i)}\NA (W_l)\Big)_{\al _j+m\al _i}
    \simeq \oplus _{l\in L,\deg W_l=\al _j+m\al _i}W_l.
    \label{eq:otNAdeg}
  \end{equation}
  Then Part~(1) of the lemma and Eqs.~\eqref{eq:KiMdeg},
  \eqref{eq:otNAdeg} give Claim~(2).
\end{proof}

The following theorem is one of the main results of \cite{p-AHS08}.

\begin{theor}\cite[Thm.\,3.12, Lemma\,3.21, Cor.\,3.17]{p-AHS08}
  Let $i\in \Ib $ and $N=([N_1],\ldots ,[N_\theta ])\in \cM _\theta $.
  Assume that $N$ is $i$-finite, and let
  $\rfl _i(N)=([N'_1],\ldots ,[N'_\theta ])$.
  
  (1) The family $\rfl _i(N)$ is $i$-finite, $\rfl _i^2(N)=N$, and
  $\cm ^N_{i j}=\cm ^{\rfl _i(N)}_{i j}$ for all $j\in \Ib $. Moreover, $\Cm
  ^N=(\cm ^N_{ij})_{i,j\in \Ib }$ is a generalized Cartan matrix.
  
  (2)
  Let $\deg N_j=\al _j$ for all $j\in \Ib $, $\deg N'_j=s_i^N(\al _j)$
  for all $j\in \Ib \setminus \{i\}$, and
  $\deg N_i^*=-\al _i$. Then
  \begin{align*}
    \NA (V(\rfl _i(N)))\simeq \cK _i^N\ot \NA (N_i^*)
  \end{align*} 
  as $\ndZ ^\theta $-graded objects in $\ydH $.

  (3) The Nichols algebras $\NA (V(N))$ and $\NA (V(\rfl _i(N)))$ have the
  same dimension.
  \label{th:R2}
\end{theor}

%\begin{defin}
%  Let $\cG (\cM _\theta )$ denote the groupoid with $\Ob (\cG (\cM _\theta
%  ))=\cM _\theta $,
%  morphism sets
%  \[\Hom (N,P)=\{(P,f,N)\,|\,f\in \Aut (\ndZ ^\theta )\}\]
%  for all $N,P\in \cM _\theta $,
%  and composition induced by the group structure of
%  $\Aut (\ndZ ^\theta )$:
%  \[ (Q,g,P)\circ (P,f,N) = (Q,gf, N)\]
%  for all $N,P,Q\in \cM _\theta $, $f,g\in \Aut (\ndZ ^\theta )$.
%  Let $M\in \cM _\theta $. Assume that $\rfl _{i_1}\cdots \rfl _{i_n}(M)\in \cM
%  _\theta $ is defined for all
%  $n\in \ndN $ and $i_1,\ldots ,i_n\in \Ib $.
%  Let
%  \[\cM _\theta (M)=\{ \rfl _{i_1}\cdots \rfl _{i_n}(M)\in \cM _\theta \,|\,
%  n\in \ndN _0, i_1,\ldots ,i_n\in \Ib \}.
%  \]
%  Let $\Wg (M)$ be the subgroupoid of $\cG (\cM _\theta )$
%  The \textit{Weyl groupoid} $\Wg (M)$ \textit{of $M$} is the subgroupoid of
%  $\cG (\cM _\theta )$
%  with objects the elements of $\cM _\theta (M)$
%  and morphisms generated by all $(\rfl _i(N),s_i^N,N)$, where
%  $i\in \Ib $ and $N\in \cM _\theta (M)$.
%  It is called the \textit{Weyl groupoid of} $M$.
%\end{defin}

%Explicitly, the morphisms of $\Wg (M)$ between $N$ and $P$ are the triples
%$(P,s,N)$, such that
%\[
%s=s_{i_n}^{\rfl _{i_{n-1}}\cdots \rfl _{i_1}(N)}\cdots
%s_{i_2}^{\rfl _{i_1}(N)}s_{i_1}^N
%\]
%and $\rfl _{i_n}\cdots \rfl _{i_2}\rfl _{i_1}(N)=P$
%for some $n\in \ndN _0$ and $i_1,\ldots ,i_n\in \Ib $.
%Indeed, the inverse of $(\rfl _i(N),s_i^N,N)\in \Hom (N,\rfl _i(N))$ is
%$(N,s_i^{\rfl _i(N)},\rfl _i(N))$, since $s_i^N$ is a reflection, and
%$s_i^N=s_i^{\rfl _i(N)}$ and $\rfl _i^2(N)=N$ by Thm.~\ref{th:R2}.

\begin{lemma}
  Let $N\in \cM _\theta $ and $i\in \Ib $. Assume that $\rsyst (N)$ is defined
  and $N$ is $i$-finite. Then $\rsyst (\rfl _i(N))$ is defined, and the map
  \begin{align*}
    s_i^N: \rsyst (N)\to \rsyst (\rfl _i(N)),\qquad
    (C,\gamma ) \mapsto (C,s_i^N(\gamma )),
  \end{align*}
  is bijective and preserves multiplicities, that is,
  \[
    \mult _N(C,\gamma) = \mult _{\rfl _i(N)}(C, s_i^N(\gamma ))
  \]
  for all $(C,\gamma) \in \rsyst (N)$.
  \label{le:invrsyst}
\end{lemma}

\begin{proof}
  Let $N=([N_1],\ldots ,[N_\theta ])\in \cM _\theta $ and $V(N)=N_1\oplus
  \cdots \oplus N_\theta $. Since $\rsyst (N)$ is defined, there is an index
  set $(L,\le )$ and a family $(W_l)_{l\in L}$ as in Def.~\ref{de:rsysN} such
  that $\NA (V(N))\simeq \ot _{l\in L}\NA (W_l)$.

  Let $\rfl _i(N)=([N'_1],\ldots ,[N'_\theta ])\in \cM _\theta $ as in
  Def.~\ref{de:cm}, and $V(\rfl _i(N))=N'_1\oplus \cdots \oplus N'_\theta $.
  By Lemma~\ref{le:Kroots} and Thm.~\ref{th:R2}(2) we obtain the decomposition
  \begin{align}
    \NA (V(\rfl _i(N)))\simeq \otimes _{l\in L,l\not=l(i)}\NA (W_l)\ot \NA
    (N_i^*)
    \label{eq:NAV'1}
  \end{align}
  as $\ndZ ^\theta $-graded objects in $\ydH $.
  We now use the group automorphism $s_i^N : \ndZ ^\theta \to \ndZ ^\theta $
  to change the gradation.
  For all $l \in L$, $l \neq l(i)$, let $W'_l = W_l$ as object in
  $\ydH $ and let  $\deg(W'_l) = s_i^N(\deg(W_l))$, and $W'_{l(i)} = N_i^*$ with
  $\deg(N_i^*) = \al _i$.
%  Let $\rfl _i(N) =([N'_1],\dots,[N'_{\theta}])$, and
%  $V(\rfl _i(N)) = N'_1 \oplus \cdots \oplus N'_{\theta}$.
  The isomorphism in \eqref{eq:NAV'1} gives an isomorphism
  \begin{equation}
    \NA(V(\rfl _i(N))) \simeq \ot_{l \in L} \NA(W'_l)
  \end{equation}
  of $\ndZ ^\theta $-graded objects in $\ydH $,
  where $\deg{N'_j} = \al _j$ for all $j \in \Ib $. Thus 
  \begin{align*}
    \rsyst _+(\rfl _i(N)) &= \{([W_l], s_i^N(\deg(W_l))\,|\,l \in L,
    l \neq l(i)\}  \cup \{([N_i^*], \al _i)\}\\
    &=s_i^N\Big(\big(\rsyst _+(N) \setminus \{([N_i],\al _i)\}\big)\cup
    \{([N_i^*],-\alpha_i)\}\Big).
  \end{align*}
  Since $\rsyst  =\rsyst _+ \cup \rsyst _-$, it follows that $\rsyst (\rfl
  _i(N))= s_i^N(\rsyst (N))$. Moreover, for all $\gamma \in \ndN _0^\theta $,
  $\gamma \neq \al _i$, and $C \in \icyH $,
  \begin{align*}
    \mult _N(C,\gamma) &= |\{ l \in L \,|\, C=[W_l], \gamma = \deg(W_l)\}|\\
    &= |\{ l \in L \,|\, C=[W'_l], s_i^N(\gamma) = \deg(W'_l)\}|\\
    &=\mult _{\rfl _i(N)}(C,s_i^N(\gamma)), 
  \end{align*}
  and $\mult _N(C,\al _i) = 1 = \mult _{\rfl _i(N)}(C^*, \al _i)
  =\mult _{\rfl _i(N)}(C, s_i^N(\al _i))$  with $C = [N_i]$.
  This proves the lemma.
\end{proof}

\begin{defin}\label{de:CS}
 For all $M\in \cM _\theta $ let
 %Assume that $\rfl _{i_1}\cdots \rfl _{i_n}(M)$ is defined for all
 %$n\in \ndN $ and $i_1,\ldots ,i_n\in \Ib $.
 \begin{align*}
 \cM _\theta (M)=&\{ \rfl _{i_1}\cdots \rfl _{i_n}(M)\in \cM _\theta \,|\,
 n\in \ndN _0, i_1,\ldots ,i_n\in \Ib \}.
 \end{align*}
% We say that $M$ \textit{defines a Cartan scheme}
% if $N$ is $i$-finite for all $N\in \cM _\theta (M)$ and $i$ in $\Ib $.
 If $N$ is $\Ib $-finite for all $N\in \cM _\theta (M)$, then let
 \begin{align*}
   \cC (M)=&(\Ib ,\cM _\theta (M),(\rfl _i)_{i\in \Ib },
   (\Cm ^N)_{N\in \cM _\theta (M)}).
 \end{align*}
\end{defin}

\begin{theor} \label{th:Cscheme}
 Let $M\in \cM _\theta $. Assume that $N$ is $\Ib $-finite for all $N\in \cM
 _\theta (M)$.
 Then $\cC (M)$ is a Cartan scheme.
\end{theor}

\begin{proof}
  This follows from Thm.~\ref{th:R2}(1).
\end{proof}

In the situation of Thm.~\ref{th:Cscheme}, $\cC (M)$ is called
the \textit{Cartan scheme of} $M$, and $\Wg (M)=\Wg (\cC (M))$ is called the
\textit{Weyl groupoid of} $M$. It is clear by construction that $\Wg (M)$ is
connected.

\begin{theor}\label{th:rsysC}
  Let $M\in \cM _\theta $. Assume that $N$ is $\Ib $-finite for all $N\in \cM
  _\theta (M)$,
  and that $\rsyst (M)$ is defined. Then
  \[ \Rwg (M)=(\cC (M), (\rsys (N))_{N\in \cM _\theta (M)}) \]
  is a root system of type $\cC (M)$, called the {\em root system
  of} $M$.
%  that is,
%  for all $i,j\in \Ib $ and $N\in \cM _\theta (M)$,
%  \begin{enumerate}
%    \item $\rsys (N)=(\rsys (N)\cap \ndN _0^{\Ib })\cup
%      -(\rsys (N)\cap \ndN _0^{\Ib })$,
%    \item $\rsys (N)\cap \ndZ \al _i=\{\al _i,-\al _i\}$,
%    % for all $i\in \Ib $ and $N\in \Ob (\Wg (M))$.
%    \item $s_i^N(\rsys (N))=\rsys (\rfl _i(N))$, and
%    \item %if $a\in \Ob (\rc )$ and $i,j\in I$ then
%          $(\rfl _i\rfl _j)^{m_{i,j;N}}(N)=N$
%          if $m_{i,j;N}:=|\rsys (N)\cap (\ndN _0\al _i+\ndN _0\al _j)|$ is
%          finite
%          and $i\not=j$.
%  \end{enumerate}
\end{theor}

If for $M\in \cM _\theta $
the assumptions of Thm.~\ref{th:rsysC} are satisfied,
%Let $M\in \cM _\theta $. Assume that $\rfl _{i_1}\cdots \rfl _{i_n}(M)$
%is defined for all $n\in \ndN $ and $i_1,\ldots ,i_n\in \Ib $,
%and that $\rsyst (M)$ is defined.
then we say that $\Rwg (M)$ \textit{is defined}. In this case $\rsyst (N)$ is
defined for all $N\in \cM _\theta (M)$ by Lemma~\ref{le:invrsyst}.

\begin{proof}
  We have to prove (R1)--(R4) from Sect.~\ref{sec:Wgrs}.
  
  (R1) follows from Eq.~\eqref{eq:rsys}, and (R2) holds by
  Lemma~\ref{le:BVdec1}. (R3) is a direct consequence of
  Lemma~\ref{le:invrsyst}.

  (R4) Let $w_0=\id \in \Aut (\ndZ ^\theta )$, and for all
  $n\in \ndN _0$ let
  \[ w_{n+1}=s_i^{\rfl _j(\rfl _i \rfl _j)^n(N)} s_j^{(\rfl _i \rfl _j)^n(N)}w_n
  \in \Aut (\ndZ ^\theta ).\]
  By \cite[Lemma\,5]{a-HeckYam08} we obtain that
  \[w_{m_{i,j;N}}(\al _i)=\al _i,\quad w_{m_{i,j;N}}(\al _j)=\al _j.
  \]
  Note that the proof of \cite[Lemma\,5]{a-HeckYam08} applies, since in our
  case $\pi _a$ in \cite{a-HeckYam08} is always the standard basis.

  Let $k\in \Ib \setminus \{i,j\}$.
  By (R3) we have
  \[w_{m_{i,j;N}}(\al _k)\in \al _k+\ndZ \al
  _i+\ndZ \al _j\cap \rsys ( (\rfl _i \rfl _j)^{m_{i,j;N}}(N)),\]
  hence
  $w_{m_{i,j;N}}(\al _k)=\al _k+n_i\al _i+n_j\al _j$ for some $n_i,n_j\in \ndN
  _0$. Since $\al _k\in \rsys ( (\rfl _i \rfl _j)^{m_{i,j;N}}(N))$,
  it follows from (R3) that
  \begin{align*}
    w_{m_{i,j;N}}^{-1}(\al _k)=&w_{m_{i,j;N}}^{-1}(w_{m_{i,j;N}}(\al _k)
    -n_i\al _i-n_j\al _j)\\
    =&\al _k-n_i\al _i-n_j\al _j\in \rsys (N).
  \end{align*}
  Hence $n_i=n_j=0$ by (R1), that is, $w_{m_{i,j;N}}=\id $.
  Then
  \[ \rsyst (N)=w_{m_{i,j;N}}(\rsyst (N))=
  \rsyst ( (\rfl _i \rfl _j)^{m_{i,j;N}}(N)) \]
  by Lemma~\ref{le:invrsyst}. Thus
  $(\rfl _i \rfl _j)^{m_{i,j;N}}(N)=N$ by Lemma~\ref{le:NrsysN}.
\end{proof}

\begin{corol}
  Let $M\in \cM _\theta $ such that $\rsyst (M)$ is defined and is finite.
  Then $\Rwg (M)$ is defined, and $\Rwg (M)$ is a root system of type $\cC
  (M)$.
\end{corol}

\begin{proof}
  By Thm.~\ref{th:rsysC}
  it suffices to prove that $\rfl _{i_n}\cdots \rfl _{i_1}(M)$
  is $\Ib $-finite for all $n\in \ndN _0$ and $i_1,\ldots ,i_n\in \Ib $.
  We proceed by induction on $n$. The case $n=0$ is trivial.
  Let $n\in \ndN _0$, $i_1,\ldots ,i_n\in \Ib $, and
  $N=\rfl _{i_n}\cdots \rfl _{i_1}(M)$. Since $\rsyst (M)$ is
  defined and is finite, Lemma~\ref{le:invrsyst} implies that $\rsyst (N)$ is
  defined and is finite. Then from Lemma~\ref{le:Kroots}(2) we obtain that
  for all $i,j\in \Ib $ with $i\not=j$ there
  exists $n_{ij}\in \ndN $ such that $(\ad _c N_i)^{n_{ij}}(N_j)=0$
  in $\NA (V(N))$,
  where $N=([N_1],\ldots ,[N_\theta ])$. Hence $N$ is $\Ib $-finite.
\end{proof}

\section{Finite root systems for Nichols algebras}
\label{sec:frs}

%Let $\theta \in \ndN $ and $\Ib =\{1,\ldots ,\theta \}$.
%Recall that $\cM _\theta =\big(\icyH \big)^\theta $. If $N=([N_1],\ldots ,[N
%_\theta ])\in \cM _\theta $, then the $[N_i]$, where $i\in \Ib $, are called
%components of $N$.
%\begin{defin}
%  Let $M=([M_1],\dots,[M_{\theta}])\in \cM _\theta $ such that $\Rwg (M)$ is
%  defined. Let
%  \begin{align*}
%    \rsyst {}\re (M) =&\{([N_i],w(\al _i))\,|\,
%  N=([N_1],\ldots ,[N_\theta ])\in \cM _\theta (M),\\
%  &w\in \Hom (N,M), i\in \Ib
%  \}.
%  \end{align*}
%\end{defin}
%
%Note that $\rsyst {}\re (M)$ consists of the pairs
%$$([N_i],s_{i_n}^{\rfl _{i_n}(M)}\cdots s_{i_1}^{\rfl _{i_1}\cdots
%\rfl _{i_n}(M)}(\al _i)),$$
%where $n\ge 0$, $i,i_1,\ldots ,i_n\in \Ib $,
%and $N=\rfl _{i_1}\cdots \rfl _{i_n}(M)$.

\begin{lemma}\label{le:realroots}
 Let $M =([M_1],\dots,[M_{\theta}])\in \cM _\theta $ such that $\Rwg (M)$ is
 defined. Let $(L,\le )$ be a totally ordered index set and
 $(W_l)_{l\in L}$ a family of \fd{} irreducible $\ndZ ^\theta $-graded
 objects in $\ydH $, such that
 \begin{equation}
  \label{Z}
  \NA (V(M))\simeq \ot _{l\in L}\NA (W_l)
 \end{equation}
 as $\ndZ ^\theta $-graded objects in $\ydH $, where $\deg(M_i) = \al _i$
 for all $i \in \Ib $.
 
 (1) For all $\gamma \in \rsys \re _+(M)$ there is
 exactly one $l(\gamma )\in L$ such that $\deg W_{l(\gamma )}=\gamma $.

 (2)
 If $N=([N_1],\ldots ,[N_\theta ])\in \cM _\theta (M)$,
 $w\in \Hom (N,M)$ and $i\in \Ib $
 such that $w(\al _i)\in \rsys \re _+(M)$, then $W_{l(w(\al _i))}\simeq N_i$.
\end{lemma}

\begin{proof}
  The decomposition \eqref{Z} exists by assumption since $\Rwg (M)$ is defined.

  Let $\gamma \in \rsys _+\re (M)$. Then there exist
  $N=([N_1],\ldots ,[N_\theta ])\in \cM _\theta (M)$,
  $w\in \Hom (N,M)$ and $i\in \Ib $ such that $w(\al _i)=\gamma $.
  Since $([N_i],\al _i)\in \rsyst (N)$, Lemma~\ref{le:invrsyst} implies that
  $([N_i],\gamma )\in \rsyst (M)$, that is, there exists $l\in L$ such that
  $\deg W_l=\gamma $ and $W_l\simeq N_i$.

  Let $l,l' \in L$ with $\gamma=\deg(W_l) = \deg(W_{l'})\in \rsys _+\re $.
  Let $N=([N_1],\ldots ,[N_\theta ]) \in \cM _\theta (M)$,
  $w\in \Hom (N,M)\subset \Hom (\Wg (M))$, and $i\in \Ib $ such that
  $w(\al _i)=\gamma $.
  Then $w^{-1}([W_l],\gamma )=([W_l],\al _i)\in \rsyst (N)$
  by Lemma~\ref{le:invrsyst}, and similarly,
  $([W_{l'}],\al_i )\in \rsyst (N)$.
  By Lemma~\ref{le:BVdec1} we obtain that
  $[W_l]=[N_i]=[W_{l'}]$.
  The second part of Lemma~\ref{le:invrsyst} implies that
  $\mult _M([W_l],\gamma) = \mult _N([N_i],\al _i) =1$. Hence $l=l'$.
  This proves (1) and (2).
\end{proof}

\begin{theor}\label{th:maingeneral}
 Let $M =([M_1],\dots,[M_{\theta}])\in \cM _\theta $ such that $\Rwg (M)$ is
 defined. Assume that $\Wg(M)$ is finite.
 Let $(L,\le )$ be a totally ordered index set and
 $(W_l)_{l\in L}$ a family of \fd{} irreducible $\ndZ ^\theta $-graded
 objects in $\ydH $, such that
 \begin{equation}
  \label{Z1}
  \NA (V(M))\simeq \ot _{l\in L}\NA (W_l)
 \end{equation}
 as $\ndZ ^\theta $-graded objects in $\ydH $, where $\deg(M_i) = \al _i$
 for all $i \in \Ib $.
 
 (1)
 The map $L \to \rsys _+ (M)$, $l \mapsto \deg(W_l)$, is bijective.
 Moreover,
 \begin{equation}
 \begin{aligned}
   \big\{[W_l],[W_l^*]\,\big|\,l\in L\big\}
   =\big\{[N_i]\,\big|\,&i\in \Ib ,\\
   &N=([N_1],\ldots ,[N_\theta ])\in \cM
   _\theta (M)\big\}.
 \end{aligned}
   \label{eq:W=N}
 \end{equation}
 
 (2) If $\gamma \in \rsys _+(M)$, then there exist
 $N=([N_1],\ldots ,[N_\theta ])\in \cM _\theta (M)$,
 $w\in \Hom (N,M)$ and $i\in \Ib $
 such that $\gamma =w(\al _i)$. In this case
 \[ W_{l(\gamma )}\simeq N_i, \]
 where $l(\gamma )$ is the unique element in $L$ with $\deg W_{l(\gamma
 )}=\gamma $.

 (3) Let $i,j \in \Ib $, $i \neq j$, and $0 \leq m \leq -\cm _{ij}^M$.
 Then there is an index $l \in L$ such that $(\ad_c M_i)^m(M_j) \simeq W_l$
 and $\deg W_l=\al _j+m\al _i$.
 In particular $(\ad_c M_i)^m(M_j)$ is irreducible in $\ydH $.
\end{theor}

\begin{proof}
  Since $\Wg (M)$ is finite, $\rsys \re (M) =\rsys (M)$ by
  Prop.~\ref{pr:allposroots}.

  We first prove (1) and (2).
  The map $L\to \rsys _+(M)$ in (1) is surjective by definition of
  $\rsys _+(M)$.
  Hence it is bijective by Lemma~\ref{le:realroots}(1) and the equality
  $\rsys \re (M) =\rsys (M)$. Then
  (2) follows from Lemma~\ref{le:realroots}(2), since
  $\rsys \re (M) =\rsys (M)$. It remains to prove Eq.~\eqref{eq:W=N}.

  Let $l\in L$ and $\gamma =\deg W_l$. Then $l=l(\gamma )$ by (1), and
  $W_l\simeq N_i$ for some $N\in \cM _\theta (M)$ and $i\in \Ib $ by (2).
  Further, if $N\in \cM _\theta (M)$ and $i\in \Ib $ then $\rfl _i(N)\in
  \cM _\theta (M)$. Since $[\rfl _i(N)_i]=[N_i^*]$, the right hand side of
  Eq.~\eqref{eq:W=N} is
  stable under passing to dual objects.
  Thus the inclusion $\subset $ holds in (3).
  Conversely, let $N\in \cM _\theta (M)$ and $i\in \Ib $. Since $\Wg (M)$ is a
  connected groupoid, there exists $w\in \Hom (N,M)$. Then
  Lemma~\ref{le:invrsyst} gives that $([N_i],w(\al _i))\in \rsyst (M)$, and
  hence $[N_i]$ is contained in the left hand side of
  Eq.~\eqref{eq:W=N}.
  
  Now we prove (3).
  Using Lemma~\ref{le:KBiso}, the isomorphism \eqref{eq:KBV} for $N=M$ of
  $\ndN _0^\theta $-graded objects in $\ydH $ implies that
  \begin{equation}
    \cK _i^M\simeq \ot _{l\in L,l\not=l(\al _i)} \NA (W_l)
    \label{eq:KiMdec}
  \end{equation}
  as $\ndN _0^\theta $-graded objects in $\ydH $. By
  \cite[Prop.\,3.6]{p-AHS08} the algebra $\cK_i^M$ is generated by the
  $\ndN _0^\theta $-homogeneous subspaces $(\ad _c M_i)^n(M_p)$ of degree $\al
  _p+n\al _i$, where $p\in \Ib \setminus \{i\}$ and $n\ge 0$, and hence
  \begin{equation}
    (\cK _i^M)_{\al _j+m\al _i}=(\ad _c M_i)^m(M_j).
%    \label{eq:KiMdeg}
  \end{equation}
  On the other hand, since $\deg W_l\in (\sum _{j=1}^\theta \ndN _0\al _j)
  \setminus \ndN _0\al _i$, and all tensor factors are $\ndN _0^\theta
  $-graded, we obtain that
  \begin{equation}
    \Big(\ot _{l\in L,l\not=l(\al _i)}\NA (W_l)\Big)_{\al _j+m\al _i}
    \simeq W_{l(\al _j+m\al _i)}.
%    \label{eq:otNAdeg}
  \end{equation}
  Then Eqs.~\eqref{eq:KiMdec}, \eqref{eq:KiMdeg} and \eqref{eq:otNAdeg} give
  the claim of Part~(3).
\end{proof}

\begin{theor}\label{th:main}
  Let $M\in \cM _\theta $. Assume that $\rsyst (M)$ is defined. Then the
  following are equivalent.
  \begin{enumerate}
    \item $\NA (V(M))$ is \fd{},
    \item 
      \begin{enumerate}
        \item $\Rwg (M)$ is defined and $\Wg (M)$ is finite, and
        \item $\NA (N_i)$ is \fd{} for all
          $N=([N_1],\ldots ,[N_\theta ])\in
          \cM _\theta (M)$ and $i\in \Ib $.
      \end{enumerate}
  \end{enumerate}
\end{theor}

\begin{proof}
  (1)$\Rightarrow $(2). Since  $\NA (V(M))$ is \fd{},
  $N$ is $\Ib $-finite for all $N\in \cM _\theta (M)$ by Thm.~\ref{th:R2}(3),
  hence $\Rwg (M)$ is defined.
  Further, $\rsys (M)$ is finite, and hence $\Wg (M)$ is finite by
  Lemma~\ref{le:finHom}. Finally, $\dim \NA (V(N))=\dim \NA (V(M))$ by
  Thm.~\ref{th:R2}(3), and hence $\NA (N_i)$ is \fd{} for all $i\in \Ib $ and
  $N=([N_1],\ldots ,[N_\theta ])\in \cM _\theta (M)$.

  (2)$\Rightarrow $(1). {}From (2a) and Lemma~\ref{le:finHom} we obtain that
 $\rsys (M)$ is finite. Hence $\NA(V(M))$ is {\fd} by (2b) and Thm.~\ref{th:maingeneral} (1) and (2).
%and $\NA (V(M))\simeq \ot _{l\in L}\NA (W_l)$
 % for a totally ordered index set $(L,\le )$ and a family $(W_l)_{l\in L}$ of
 %\fd{} irreducible $\ndN _0^\theta $-graded objects in $\ydH $.
 % Since the homogeneous components of $\NA (V(M))$ are \fd{}, it follows that
 %$L$ is finite. By (2) and Lemma~\ref{le:comp}, $\NA (W_l)$ is \fd{} for all
 %$l\in L$. Thus (1) follows.
\end{proof}

Recall that $\rsyst (M)$ is defined in particular if all tensor powers of
$V(M)$ are semisimple.

Let $M\in \cM _\theta $. Then $M$ is called \textit{standard}, see
\cite[Def.\,3.23]{p-AHS08}, if $N$ is $\Ib $-finite
and $\Cm ^N=\Cm ^M$ for all $N\in \cM _\theta (M)$.

\begin{corol}
  Let $M\in \cM _\theta $. Assume that $\rsyst (M)$ is defined, and that $M$
  is standard. Then $\Rwg (M)$ is defined, and the
  following are equivalent.
  \begin{enumerate}
    \item $\NA (V(M))$ is \fd{},
    \item 
      \begin{enumerate}
        \item $\Cm ^M$ is a Cartan matrix of finite type,
        \item $\NA (N_i)$ is \fd{} for all
          $N=([N_1],\ldots ,[N_\theta ])\in
          \cM _\theta (M)$ and $i\in \Ib $.
      \end{enumerate}
  \end{enumerate}
\end{corol}

\begin{proof}
 This follows from Thm.~\ref{th:main} and Cor.~\ref{co:standard}.
 %(1)$\Rightarrow $(2). Since $M$ is standard, the image of the natural map
 %$\Hom (\Wg (M))\to \Aut (\ndZ ^\theta )$, $(P,f,N)\mapsto f$, is
 %the group generated by $s_1^M,\ldots ,s_\theta ^M$, which is the Weyl
 %group of the Kac-Moody Lie algebra associated to the Cartan matrix $\Cm ^M$.
 %Thus (2) follows from Thm.~\ref{th:main}.
%(2)$\Rightarrow $(1). Since $\Cm ^M$ is of finite type, there exist
%Allgemeiner beweisen fuer standard root systems of type $\cC $.
\end{proof}

%\section{Symbols and terminology}
%
%\begin{tabular}{l|l}
%  $\cC (M)$ & Def.~\ref{de:CS}\\
%  $\rsyst _{(\pm )}(N)$, $\rsys _{(\pm )}(N)$ & Def.~\ref{de:rsysN}\\
%  $\cM _\theta $ & above Def.~\ref{de:rsysN}\\
%  $\cM _\theta (M)$ & Def.~\ref{de:CS}\\
%  $\rsyst (N)$ is defined & Def.~\ref{de:rsysN}\\
%  $\rfl _i(N)$ is defined & Def.~\ref{de:cm}\\
%  $\Rwg (N)$ is defined & below Thm.~\ref{th:rsysC}\\
%  Cartan scheme & Sect.~\ref{sec:Wgrs}\\
%  connected groupoid & above Lemma~\ref{le:finrs}\\
%  generalized Cartan matrix & Sect.~\ref{sec:Wgrs}\\
%  root system of type $\cC $ & Sect.~\ref{sec:Wgrs}\\
%  roots (positive, negative, real) & Sect.~\ref{sec:Wgrs}\\
%  roots of $N$ (positive, negative, real) & Def.~\ref{de:rsysN}\\
%  standard Cartan scheme & below Lemma~\ref{le:finrs2}\\
%  Weyl groupoid of $\cC $ (of $\Rwg $) & Sect.~\ref{sec:Wgrs}
%\end{tabular}

\section{Applications for finite groups}
\label{sec:app}

In this section let $\fie $ be an algebraically closed field of characteristic
$0$, and let $H=\fie G$ be the group algebra of a finite group $G$.
For any $g\in G$ let $\cc g$ denote the conjugacy class of $g$ and $G^g$ the
centralizer of $g$ in $G$. Let $\Ad{}{}$ denote the adjoint action in
$G$,
that is, $\Ad{g}{h}=ghg^{-1}$ for all $g,h\in G$.
The category $\ydH $ will be denoted by $\ydG $.

Let $g\in G$ and $M$ an irreducible $G^g$-module.
Then $kG\ot _{kG^g}M$ is an irreducible object in $\ydG $, where
\begin{align*}
  h\cdot (h'\ot m)=hh'\ot m,\quad
  \delta (h\ot m)=hgh^{-1}\ot (h\ot m)
\end{align*}
for all $h,h'\in G$, $m\in M$.
Any irreducible object in $\ydG $ arises in
this way. The Yetter-Drinfeld module $V=kG\ot _{kG^g}M$ can be written as
\begin{align*}
  V=\oplus _{s\in \cc g}V_s,\quad V_s=\{v\in V\,|\,\delta (v)=s\ot v\},
\end{align*}
and $V_g=1\ot M$.
For all $s\in \cc g$, $V_s$ is an irreducible $G^s$-module, and
$h\cdot V_s =V_{\Ad{h}{s}}$ for all $h\in G$.
There exists $q_V\in \fie ^*$ such that
$s\cdot v=q_V v$ for all $v\in V_s$, $s\in \cc g$.

Let $g,h\in G$. We say that $\cc g$ and $\cc h$ \textit{commute}, if
$st=ts$ for all $s\in \cc g$, $t\in \cc h$.

\begin{propo}\label{pr:ad=0}
  Let $g,h\in G$ and $V=\oplus _{s\in \cc g}V_s$, $W=\oplus _{t\in \cc h}W_t$
  be irreducible objects in $\ydG $.

  (1) If $(\ad _c V)(W)=0$ in $\NA (V\oplus W)$,
  then $\cc g$ and $\cc h$ commute.

  (2) If $(\ad _c V)^2(W)=0$ in
  $\NA (V\oplus W)$, then $\cc g$ commutes with itself or with $\cc h$.
\end{propo}

\begin{proof}
  (1) By Prop.~\ref{pr:Tn}, $(\ad _c V)(W)$ is isomorphic to $(\id -c^2)(V\ot
  W)$ in $\ydG $.
  If $s\in \cc g$, $t\in \cc h$, $st\not= ts$, then
  \[ c^2(V_s\ot W_t)=V_{\Ad{st}{s}}\ot W_{\Ad{s}{t}},\]
  and hence $(\id -c^2)(V_s\ot W_t)\not=0$.

  (2) It suffices to show that if $(\ad _cV)^2(W)=0$
  and $gh\not=hg$, then $rg=gr$ for all $r\in \cc g\setminus
  \{g\}$.
  We note that if $r\in \cc g$, $v\in V_r$, $v'\in V_g$, $w\in W_h$,
  $v\ot v'\ot w\not=0$, then $(S_2\ot \id)T_2(v\ot v'\ot w)$ is the sum of
  eight nonzero homogeneous terms of degrees
  \begin{align*}
    &(r,g,h), (\Ad{r}{g},r,h), (r,\Ad{gh}{g}, \Ad{g}{h}),
    (\Ad{rgh}{g},r,\Ad{g}{h}),\\
    &(\Ad{r}{g},\Ad{rh}{r},\Ad{r}{h}),(\Ad{rgh}{r},\Ad{r}{g},\Ad{r}{h}),\\
    &(\Ad{rgh}{g},\Ad{rghg^{-1}}{r},\Ad{rg}{h}),
    (\Ad{rgh}{r},\Ad{rgh}{g},\Ad{rg}{h}).
  \end{align*}
  On the other hand, $(S_2\ot \id)T_2(v\ot v'\ot w)=0$
  by Prop.~\ref{pr:Tn} and the assumption
  $(\ad _c V_r)(\ad _cV_g)(W_h)=0$. The vanishing of the sum of these
  eight terms gives
  information about the corresponding degrees.

  Let first $r\in \cc g$ such that $rh\not=hr$ and $r\not=g$.
  By comparing degrees we get $\Ad{rg}{h}=h$ and
  $\Ad{rgh}{g}=r$. Hence $r=\Ad {gh}{g}$ is uniquely determined.
  Since $\Ad{h}{g}$ and $\Ad{h^{-1}}{g}$ do not commute with $h$
  and are different from $g$, we
  get
  \begin{align}\label{eq:r}
    r=\Ad{h^{-1}}{g}=\Ad{h}{g}=\Ad{gh}{g}.
  \end{align}
  Thus $g=\Ad{h^{-1}gh}{g}=\Ad{r}{g}$ by the last and first equations in
  Eq.~\eqref{eq:r}, that is, $rg=gr$.
 
  Let now $r\in \cc g$ with $rh=hr$. Then $r\not=g$, and
  by comparing degrees one gets
  $\Ad{rg}{h}=\Ad{g}{h}$ and $\Ad{r}{g}=g$. This proves the claim.
\end{proof}

Let $\mathcal{F}(G)$ be the set of all conjugacy classes $\cc{}$ of $G$
such that $\dim \NA (V)<\infty $ for some
$V=\oplus _{s\in \cc{}}V_s \in \ydG $.

\begin{theor}\label{th:Uirred}
  Assume that any two conjugacy classes in $\mathcal{F}(G)$
  do not commute.
  Let $U\in \ydG $. If $\NA (U)$ is \fd{}, then $U$ is irreducible
  in $\ydG $.
\end{theor}

\begin{proof}
  Recall that the category $\ydG $ is semisimple, and that embeddings of
  Yetter-Drinfeld modules induce embeddings of the corresponding Nichols
  algebras. Hence it suffices to prove that $\NA (V\oplus W)$ is infinite
  dimensional for all irreducible objects $V,W\in \ydG $.

  Let $g,h\in G$,
  $V=\oplus _{s\in \cc g}V_s$, $W=\oplus _{t\in \cc h}V_t$ irreducible objects
  in $\ydG $
  and $M=([V],[W])\in \cM _2$.
  Assume that $\dim \NA (V\oplus W)<\infty $.
  Then $\dim \NA (V)$, $\dim \NA (W)<\infty $, and hence $\cc g,\cc h\in
  \mathcal{F}(G)$.

  If $a_{12}^M=0$, then $(\ad _cV)(W)=0$, and hence Prop.~\ref{pr:ad=0}(1)
  gives that $\cc g$ and $\cc h$ commute. If $a_{12}^M=-1$, then
  $(\ad _cV)^2(W)=0$, and Prop.~\ref{pr:ad=0}(2)
  gives that $\cc g$ commutes with $\cc g$ or with $\cc h$.
  Thus the assumption in
  the theorem yields that $a^M_{12}<-1$.
  Similarly, $a^M_{21}<-1$.
  By Thm.~\ref{th:R2}(3),
  $\dim \NA (N_1\oplus N_2)=\dim \NA (M_1\oplus M_2)$ for all
  $N\in \cM _2(M)$, $N=([N_1],[N_2])$, and hence the above arguments give that
  $a^N_{12},a^N_{21}<-1$ for all $N\in \cM _2(M)$.
  Since $\Hom (\Wg (M))$ is finite by Thm.~\ref{th:main},
  the theorem follows from Prop.~\ref{pr:aij=-1}.
\end{proof}

Recall that if $V=V_1$, then $\dim \NA (V)=\infty $, since the characteristic
of $\fie $ is $0$. Thus $\cc 1\notin \mathcal{F}(G)$.

\begin{corol} \label{co:sgrp}
  Let $G$ be a nonabelian simple group.
  Let $U\in \ydG $. If $\NA (U)$ is \fd{}, then $U$ is irreducible
  in $\ydG $.
\end{corol}

\begin{proof}
  Let $\cc{}'$, $\cc{}''$ be two commuting conjugacy classes of $G$. Then the
  subgroups $G'=\langle \cc{}'\rangle $,
  $G''=\langle \cc{}''\rangle $ of $G$ are normal. Since $G$ is simple,
  $G',G''$ are either $1$ or $G$. Since $\cc{}'$, $\cc{}''$ commute,
  $[G',G'']=1$. Hence $G'=1$ or $G''=1$, and hence $\cc{}'=1$ or
  $\cc{}''=1$.
  But $1\notin \mathcal{F}(G)$, and the claim follows from
  Thm.~\ref{th:Uirred}.
\end{proof}

\begin{corol}\label{co:Sn}
  Let $n\in \ndN $, $n\ge 3$, and assume that $G=\mathbb{S}_n$ is the
  symmetric group.
  Let $U\in \ydG $. If $\NA (U)$ is \fd{}, then $U$ is irreducible
  in $\ydG $.
\end{corol}

\begin{proof}
  For $n\ge 5$ the group $\mathbb{S}_n$ has a simple group of index two.
  Along the lines of the proof of Cor.~\ref{co:sgrp}
  it is easy to show that $\mathbb{S}_n$, $n\ge 5$,
  does not possess commuting nontrivial (i.\,e.~different from $\{1\}$)
  conjugacy classes.
  In the usual cycle notation, the only pairs of commuting nontrivial
  conjugacy classes of $\mathbb{S}_n$ are
  $(\cc{(123)},\cc{(123)})$ for $n=3$ and
  $(\cc{(12)(34)},\cc{(12)(34)})$ for $n=4$.
  By \cite[Thm.~1]{a-AndrZhang07},
  $\cc{(1\,2\,3)}\notin \mathcal{F}(\mathbb{S}_3)$
  and
  $\cc{(1\,2)(3\,4)}\notin \mathcal{F}(\mathbb{S}_4)$.
  Thus the claim follows from Thm.~\ref{th:Uirred}.
\end{proof}
 
For the dihedral groups $\mathbb{D}_n$, $n$ odd, an alternative proof of
\cite[Thm.~4.8]{p-AHS08} can be given as another application of
Thm.~\ref{th:Uirred}.

We continue with some other consequences of our theory which hold for all
finite groups.

\begin{propo}\label{pr:appadVW}
  Let $g,h\in G$ and $V=\oplus _{s\in \cc g}V_s$, $W=\oplus _{t\in \cc h}W_t$
  be irreducible objects in $\ydG $. Assume that $(\ad _cV)(W)$ is
  irreducible.
  
  (1) $stst=tsts$ for all $s\in \cc g$, $t\in \cc h$.

  (2) There is at most one double coset $G^hxG^g$ in $G^h{\setminus }G/G^g$,
  $x\in G$,
  such that $(\Ad{x}{g})h\not=h(\Ad{x}{g})$.
  If it exists, then $(\ad _cV)(W)\not=0$ and $q_{(\ad _cV)(W)}=-q_Vq_W$.
\end{propo}

\begin{proof}
  By Prop.~\ref{pr:Tn}, $(\ad _cV)(W)$ is isomorphic to $(\id -c^2)(V\ot W)$
  in $\ydG $. Let $s\in \cc g$, $t\in \cc h$, $v\in V_s$, and
  $w\in W_t$. Then
  \[ c^2(v\ot w)=c((s\cdot w)\ot v)=sts^{-1}\cdot v\ot s\cdot w. \]
  Since $s\cdot v=q_V v$ and $t\cdot w=q_W w$, we obtain that
  \begin{align}
    (\id -c^2)(v\ot w)=v\ot w-q_V^{-1}q_W^{-1} st\cdot (v\ot w).
    \label{eq:id-c2}
  \end{align}
  Since $(\ad _cV)(W)$ is irreducible,
  there exists $q\in \fie ^*$ such that
  \begin{align}
    st\cdot (\id-c^2)(v\ot w)=q(\id -c^2)(v\ot w).
    \label{eq:stact}
  \end{align}
  Comparing degrees we obtain that $\Ad{st}{t}=t$ or $\Ad{(st)^2}{t}=t$.
  Thus $(st)^2=(ts)^2$. This proves (1).

  If $X\subset G$ is a set of double coset representatives for $G^h{\setminus}
  G/G^g$, then
  \[ V\ot W=\oplus _{x\in X}\fie G\cdot ((x\cdot V_g)\ot W_h) \]
  is a decomposition into Yetter-Drinfeld modules over $G$. By
  Eq.~\eqref{eq:id-c2},
  \[ (\id -c^2)(V\ot W)=\oplus _{x\in X}\fie G\cdot (\id -c^2)
  ((x\cdot V_g)\ot W_h) \]
  is a decomposition into Yetter-Drinfeld modules over $G$.
  Let $x\in X$. Assume that $(\Ad{x}{g})h\not=h(\Ad{x}{g})$.
  Then $(\id -c^2)((x\cdot V_g)\ot W_h)\not=0$ by Eq.~\eqref{eq:id-c2}.
  Since $(\ad _cV)(W)$ is irreducible by assumption,
  Prop.~\ref{pr:Tn} implies that $(\id -c^2)(V\ot W)$ is irreducible.
  Thus $(\id -c^2)((y\cdot V_g)\ot W_h)=0$ and
  hence $(\Ad{y}{g})h=h(\Ad{y}{g})$ for all $y\in X\setminus \{x\}$.
  The claim $q_{(\ad _cV)(W)}=-q_Vq_W$ follows from Eq.~\eqref{eq:stact}
  for $s=\Ad{x}{g}$, $t=h$.
\end{proof}

%By Thms.~\ref{th:maingeneral}(3) and \ref{th:main}
%the above result applies if $\NA (V\oplus W)$ is \fd{}.

\begin{theor}\label{th:stst}
  Let $g,h\in G$ and $V=\oplus _{s\in \cc g}V_s$, $W=\oplus _{t\in \cc h}W_t$
  be irreducible objects in $\ydG $. If $\NA (V\oplus W)$ is \fd{}, then
  the following hold.

  (1) For all $s\in \cc g$ and $t\in  \cc h$, $(st)^2=(ts)^2$.

  (2) There is at most one double coset
  $G^hxG^g$
  in $G^h{\setminus }G/G^g$,
  $x\in G$,
  such that $(\Ad{x}{g})h\not=h(\Ad{x}{g})$.

  (3) For all $s\in \cc g$, $t\in \cc h$ with $st\not=ts$
  there is an irreducible object $U=\oplus _{r\in \cc{st}}U_r\in \ydG $
  satisfying $q_U=-q_Vq_W$ and $\dim \,\NA (U)<\infty $.
\end{theor}

\begin{proof}
  By Thms.~\ref{th:main} and \ref{th:maingeneral}(3), $(\ad _cV)(W)$ is either
  $0$ or irreducible.
  Then (1) and (2) follow from  Prop.~\ref{pr:appadVW}(2).

  (3) Assume that $s\in \cc g$, $t\in \cc h$ such that $st\not=ts$.
  Then $U=(\ad _cV_s)(W_t)\not=0$ by Eq.~\eqref{eq:id-c2}, hence $U=
  \oplus _{r\in \cc{st}}U_r$ is irreducible. Therefore $q_U=-q_Vq_W$ by
  Prop.~\ref{pr:appadVW}(2), and $\NA (U)$ is \fd{} by
  Thm.~\ref{th:maingeneral}
  and since $\NA (V\oplus W)$ is \fd{}.
\end{proof}

%\bibliographystyle{amsalpha}
%\bibliography{quantum}

\providecommand{\bysame}{\leavevmode\hbox to3em{\hrulefill}\thinspace}
\providecommand{\MR}{\relax\ifhmode\unskip\space\fi MR }
% \MRhref is called by the amsart/book/proc definition of \MR.
\providecommand{\MRhref}[2]{%
  \href{http://www.ams.org/mathscinet-getitem?mr=#1}{#2}
}
\providecommand{\href}[2]{#2}

\end{document}